\documentclass{amsart}
\usepackage[utf8]{inputenc}
\usepackage[margin=1.3in]{geometry}

\usepackage{amsmath, amsfonts, amssymb, amsthm}
\usepackage{setspace}
\usepackage{enumerate}
\usepackage{graphicx}
\usepackage{xcolor}
\usepackage{soul}
\usepackage{url}
\usepackage{cite}
\usepackage[colorlinks,citecolor=blue]{hyperref}
\usepackage{verbatim}

\newtheorem{theorem}{Theorem}[section]
\newtheorem{lemma}[theorem]{Lemma}
\newtheorem{proposition}[theorem]{Proposition}
\newtheorem*{theorem:main}{Main Theorem}
\newtheorem*{thmA}{Theorem A}
\newtheorem*{thmB}{Theorem B}
\newtheorem*{thmC}{Theorem C}

\theoremstyle{definition}
\newtheorem{definition}[theorem]{Definition}
\newtheorem{example}[theorem]{Example}

\newtheorem{notation}[theorem]{Notation}

\theoremstyle{remark}
\newtheorem{remark}[theorem]{Remark}

\numberwithin{equation}{section}

\newcommand{\free}{\mathbb{F}} 
\newcommand{\fltr}{\emptyset = G_0 \subset G_1 \subset \dots \subset G_K = G}
\newcommand{\Rfltr}{\emptyset = G_0 \subset G_1 \subset \dots \subset G_r = G}
\newcommand{\rfa}{\mathcal{FF}(\free,\mathcal{A})} 
\newcommand{\oo}{\Phi} 
\newcommand{\ffa}{\mathcal{A}}
\newcommand{\rc}{\mathcal{RC}(\mathcal{A})}
\newcommand{\prc}{\mathbb{P}\mathcal{RC}(\mathcal{A})}
\newcommand{\mrc}{\mathcal{MRC}(\mathcal{A})}
\newcommand{\relOut}{\operatorname{Out}(\mathbb{F}, \ffa)} 
\newcommand{\lamination}{\Lambda^+_{\oo}} 
\newcommand{\Rlamination}{\Lambda^-_{\oo}} 
\newcommand{\PFevalue}{\lambda_{\oo}}
\newcommand{\StableCurrent}{\eta^+_{\oo}}
\newcommand{\UnstableCurrent}{\eta^-_{\oo}}
\newcommand{\RelativeBasis}{\mathfrak{B}_{\ffa}} 
\newcommand{\relCVClo}{\overline{\mathbb{P}\mathcal{O}(\free, \ffa)}} 
\newcommand{\relCV}{\mathbb{P}\mathcal{O}(\free, \ffa)}
\newcommand{\RelCVClo}{\overline{\mathcal{O}(\free, \ffa)}}
\newcommand{\RelCV}{\mathcal{O}(\free, \ffa)}
\newcommand{\StableTree}{T^{+}_{\oo}} 
\newcommand{\UnstableTree}{T^{-}_{\oo}}
\newcommand{\StableTreePhi}{T^{+}_{\phi}} 
\newcommand{\PMTrees}{T^{\pm}_{\oo}}
\newcommand{\PMlaminations}{\Lambda^{\pm}_{\oo}}
\newcommand{\MPlaminations}{\Lambda^{\mp}_{\oo}}
\newcommand{\PMCurrents}{\eta^{\pm}_{\oo}}
\newcommand{\MPCurrents}{\eta^{\mp}_{\oo}}
\newcommand{\RelTTTree}{T_G}
\newcommand{\AbsTree}{T_{G^{\p}}}
\newcommand{\dd}{\mathcal{D}}
\newcommand{\ee}{\mathcal{E}}

\newcommand{\op}{\operatorname}
\newcommand{\alert}{\textcolor{red}}
\newcommand{\la}{\langle}
\newcommand{\ra}{\rangle}
\newcommand{\p}{\prime}

\begin{document}

\title{Loxodromic elements for the relative free factor complex}
\author{Radhika Gupta}
\address{Department of Mathematics, Technion, Haifa, Israel}
\curraddr{}
\email{radhikagup@technion.ac.il}
\thanks{The author was partially supported by the NSF grant of Mladen Bestvina (DMS-1607236)}
\date{}
\dedicatory{}
\begin{abstract}
In this paper we prove that a fully irreducible outer automorphism relative to a non-exceptional free factor system acts loxodromically on the relative free factor complex as defined in \cite{HM:RelativeComplex}. We also prove a north-south dynamic result for the action of such outer automorphisms on the closure of relative outer space.  
\end{abstract}
\maketitle
\section{Introduction}
The study of the outer automorphism group $\op{Out}(\free)$ of a free group $\free$ of rank $n$ is highly motivated by the parallels with the mapping class group $\op{MCG}(\Sigma)$ of a surface $\Sigma$. $\op{MCG}(\Sigma)$ acts on a simplicial complex called the \emph{curve complex} $\mathcal{C}(\Sigma)$. In 1999, Masur and Minsky \cite{MM:CurveComplex} showed that $\mathcal{C}(\Sigma)$ is hyperbolic and since then it has played a crucial role in understanding $\op{MCG}(\Sigma)$. Some remarkable applications include rigidity results for $\op{MCG}(\Sigma)$, bounded cohomology for subgroups of $\op{MCG}(\Sigma)$ and finite asymptotic dimension for $\op{MCG}(\Sigma)$. Several analogues of the curve complex for $\op{Out}(\free)$ have been defined and proven to be hyperbolic, like the \emph{free factor complex}, the \emph{free splitting complex} and the \emph{cyclic splitting complex}. But none of them have proven to be as useful as the curve complex. 

For instance, when a mapping class group element acts on $\mathcal{C}(\Sigma)$ with a fixed point, that is, it fixes a curve $\alpha$, then one can look at its action on the curve complex of the subsurface given by the complement of $\alpha$. Thus we can understand mapping class group elements by an inductive process. On the other hand, consider an outer automorphism which fixes a free factor $A$ in the free factor complex of $\free$. Since the complement of $A$ in $\free$ is not well defined one cannot pass to the free factor complex of a free group of lower rank. 

In \cite{HM:RelativeComplex}, Handel and Mosher define \emph{free factor complex relative to a free factor system} $\rfa$ which is an $\op{Out}(\free)$-analog of the curve complex for a subsurface. They also prove that these relative complexes are hyperbolic for \emph{non-exceptional} free factor systems. The \emph{exceptional} free factor systems are certain ones for which $\rfa$ is either empty or zero-dimensional. They can be enumerated as follows (see Section~\ref{subsec:Relative free factor complex}): $\ffa = \{ [A_1], [A_2] \}$ with $\free = A_1 \ast A_2$,  $\ffa = \{[A]\}$ with $\free = A \ast \mathbb{Z}$ and $\ffa = \{ [A_1], [A_2], [A_3]\}$ with $\free =A_1 \ast A_2\ast A_3$. Here $[.]$ denotes the conjugacy class of a free factor. Since any free factor system of the free group of rank 2 is exceptional, we will work with free groups of rank at least 3. 

Our main theorem is a relative version of a result of \cite{MM:CurveComplex} that a mapping class group element acts \emph{loxodromically}, that is with positive translation length, on the curve complex if and only if it is a pseudo-Anosov homeomorphism. Let $\ffa$ be a free factor system. Let $\relOut$ be the subgroup of $\op{Out}(\free)$ containing outer automorphisms that preserve $\ffa$. After passing to a finite index subgroup we can assume that each conjugacy class of a free factor in $\ffa$ is invariant under the elements of $\relOut$. An outer automorphism $\oo \in \relOut$ is \emph{fully irreducible relative to $\ffa$} if no power of $\oo$ preserves a non-trivial free factor system of $\free$ properly containing $\ffa$.

\begin{thmA}\label{T:main}
Let $\ffa$ be a non-exceptional free factor system in a finite rank free group $\free$ of rank at least 3 and let $\oo \in \op{Out}(\free, \ffa)$. Then $\oo$ acts loxodromically on $\rfa$ if and only if $\oo$ is fully irreducible relative to $\ffa$. 
\end{thmA}
Alternative proof of Theorem A was announced by Handel and Mosher in \cite{HM:RelativeComplex} by generalizing their arguments, which use weak attraction theory, for loxodromic elements for the free splitting complex. Independently, Guirardel and Horbez \cite{GH:Loxodromic} have an alternate proof of Theorem A using the boundary of the relative free factor complex. Guirardel and Horbez posted a paper on the arXiv (after our paper was submitted for publication) \cite{GH:LaminationsFreeProducts} which is the first step towards a description of the Gromov boundary of the free factor graph of a free product. 

\subsection*{Pseudo-Anosov and the curve complex} We give an outline of how to prove that a pseudo-Anosov homeomoprhism acts loxodromically on the curve complex to illustrate the strategy we use to prove Theorem A. The following proof is due to Bestvina and Fujiwara \cite[Proposition 11]{BK:BoundedCohomology}. 

Let $\Sigma$ be a closed surface of genus $g$ and let $\mathcal{C}(\Sigma)$ be the curve complex. Let $\Psi$ be a pseudo-Anosov mapping class group element. Let $\Lambda^+$ and $\Lambda^-$ be the attracting and repelling measured laminations associated to $\Psi$. Let $\mathcal{PML}(\Sigma)$ be the space of all projective measured laminations, which contains the curve complex as a subset. We will need the following facts: 
\begin{itemize}
\item The pseudo-Anosov $\Psi$ acts on $\mathcal{PML}(\Sigma)$ with uniform north-south dynamics, that is, there are two fixed points $\Lambda^+$ and $\Lambda^-$ and any compact set not containing $\Lambda^-(\Lambda^+)$ converges to $\Lambda^+$($\Lambda^-$) under $\Psi$($\Psi^{-1}$)-iterates. 
\item The intersection number $i(\cdot, \cdot)$ between two curves in the curve complex extends to a continuous, symmetric bilinear form $i:\mathcal{PML}(\Sigma) \times \mathcal{PML}(\Sigma) \to \mathbb{R}$.
\item The fixed points $\Lambda^+$ and $\Lambda^-$ are uniquely self-dual, that is, $i(\Lambda^{\pm}, \mu) = 0$ if and only if $\mu = \Lambda^{\pm}$.  
\end{itemize}

If $U$ is a neighborhood of $\Lambda^+$ then there exists a neighborhood $V$ of $\Lambda^+$, such that $V \subset U$ and if $a \in U^C$, $b\in V$ then $i(a,b)>0$. Indeed, if this is not true then we can find a sequence of neighborhoods $U \supset V_1 \supset V_2 \supset \ldots$ and curves $a_i \in U^C$ and $b_i \in V_i$ such that $\{b_i\}$ converges to $\Lambda^+$, $\{a_i\}$ converges to $a \neq \Lambda^+$ and $i(a_i, b_i) = 0$. But by continuity of the intersection number, $i(a_i, b_i)$ converges to $i(a, \Lambda^+)$ which is not zero. We call such a pair a UV-pair. Now consider a sequence of nested neighborhoods of $\Lambda^+$, $U_0 \supset U_1 \supset U_2 \supset U_3 \supset\ldots \supset U_{2N}$ for some $N>0$, such that the following hold:
\begin{itemize}
 \item $(U_i, U_{i+1})$ is a UV-pair for all $0 \leq i <2N$. 
\item $\exists \,k>0$ such that for all $0 \leq i < 2N$, $\Psi^k(U_i) \subset U_{i+1}$ 
\end{itemize}
Let $a$ be a curve such that $a \in U_0$ and $a \notin U_1$. Given $\alpha \in U_i^C$ such that $i(\alpha, \beta) = 0$ then $\beta \in U_{i+1}^C$. Thus we get that $d(a, \Psi^{2Nk}(a)) > N$ in the curve complex.

The above proof strategy can also be employed to prove that a fully irreducible outer automorphism acts loxodromically on the free factor complex (original proof in \cite{BF:HyperbolicComplex}). Though in this case we need north-south dynamics on a certain space of measured currents (\cite{Martin}, \cite{U:HyperbolicIWIP}), north-south dynamics on the closure of outer space (\cite{LL:NSDynamics}) and an intersection number between measured currents and $\free$-trees in the closure of outer space (\cite{KL:IntersectionForm}). We will refer to the case of the fully irreducible outer automorphism as the `absolute case'.

\subsection*{Proof outline} We give an overview of how we generalize the key ingredients mentioned above to the relative setting in order to prove Theorem A. Let $\free = A_1 \ast \cdots \ast A_k \ast F_N$ be a free factor decomposition of $\free$ and let $\ffa = \{[A_1], \ldots, [A_k] \}$, $k \geq 0$ be a free factor system, where $[\cdot]$ denotes the conjugacy class. Let $\zeta(\ffa) = k+N$. The free factor systems $\emptyset$ and $\{[\free]\}$ are called \emph{trivial free factor systems}. Let $\oo$ be a fully irreducible outer automorphism relative to $\ffa$. 

In \cite{G:RelativeCurrents} we define \emph{relative currents} and prove a north-south dynamic result on a subspace $\mrc$ of the space of projective relative currents. See Section~\ref{subsec:RelativeCurrents} for definitions.  

\begin{thmB}[{\cite{G:RelativeCurrents}}]\label{T:NS for currents}
Let $\ffa$ be a non-trivial free factor system of $\free$ such that $\zeta(\ffa) \geq 3$. Let $\oo \in \relOut$ be fully irreducible relative to $\ffa$. Then $\oo$ acts on $\mrc$ with uniform north-south dynamics: there are only two fixed points $\StableCurrent$ and $\UnstableCurrent$ and 
for every compact set $K$ of $\mrc$ that does not contain $\UnstableCurrent$ (rep. $\StableCurrent$), and for every neighborhood $U$ (resp. $V$) of $\StableCurrent$ (resp. $\UnstableCurrent$), there exists $N \geq 1$ such that for all $n \geq N$ we have $\oo^n(K) \subseteq U$ (resp. $\oo^{-n}(K) \subseteq V$).    
\end{thmB}

In \cite{GL:RelativeOuterSpace}, Guirardel and Levitt define relative outer space for a countable group that splits as a free product. For the group $\free$ and a free factor system $\ffa$ we denote the relative outer space by $\relCV$. In Section~\ref{sec:NorthSouth} we prove the following theorem:

\begin{thmC}\label{T:NS for trees}
Let $\ffa$ be a non-trivial free factor system of $\free$ such that $\zeta(\ffa) \geq 3$. Let $\oo \in \relOut$ be fully irreducible relative to $\mathcal{A}$. Then $\oo$ acts on $\relCVClo$ with uniform north-south dynamics: there are two fixed points $\StableTree$ and $\UnstableTree$ and for every compact set $K$ of $\relCVClo$ that does not contain $\UnstableTree$ (resp. $\StableTree$) and for every open neighborhood $U$ (resp. $V$) of $\StableTree$ (resp. $\UnstableTree$), there exists an $N\geq 1$ such that for all $n \geq N$ we have $K\oo^n \subseteq U$ (resp. $K\oo^{-n} \subseteq V$) 
\end{thmC}
Let $\phi^{\p}: G^{\p} \to G^{\p}$ be a relative train track representative of an outer automorphism $\oo$ which is fully irreducible relative to $\ffa$, here $G^{\p} \in CV_n$. Let $G$ be the graph obtained by equivariantly collapsing the maximal invariant proper subgraph of $G^{\p}$. Then the set of vertex stabilizers of $G$ is exactly $\ffa$. We follow the techniques of \cite{LL:NSDynamics} to prove Theorem C. The key difference in generalizing the $\mathcal{Q}$ map arguments for trees with dense orbits arises from the fact that unlike the absolute case, the dual lamination of the stable tree in the relative setting contains lines which may not be leaves of the repelling lamination. These lines are \emph{diagonal leaves} that come from concatenating certain rays which we call \emph{eigenrays} based at a vertex with non-trivial stabilizer in $G$ (see Section~\ref{subsec:Q map} for definition). Moreover, unlike the case of a fully irreducible outer automorphism the Whitehead graph of the attracting lamination (Definition~\ref{D:WhGraph}) of a relative fully irreducible outer automorphism may not be connected at a vertex of $G^{\p}$.  We define a transverse covering for the universal cover of $G$ to understand these differences from the absolute case. We also define a relative Whitehead graph (Section~\ref{subsec:Relative Wh graph}) and show that it is connected to prove convergence for simplicial trees in $\relCVClo$.

It turns out that the intersection number between a rational relative current $\eta_{g}$ and a relative tree $T$ in relative outer space defined as the translation length of $g$ in $T$ cannot be extended continuously to the product of the space of relative currents and the closure of relative outer space. See Section~\ref{sec:IntersectionForm} for an example due to Camille Horbez and \cite[Theorem 6]{GH:LaminationsFreeProducts}. In Section~\ref{sec:IntersectionForm}, we give a definition of an intersection form based on the zero pairing criterion in \cite{KL:ZeroLength} which is sufficient for our purposes. If we have a sequence of relative currents $\eta_n$ converging to $\eta$ and a sequence of trees $T_n$ converging to $T$ such that $\eta_n$ is dual to $T_n$ under the intersection form then in general it is not true that $\eta$ is dual to $T$. However, we show that this is true when either $T$ has dense orbits or support of $\eta$ is birecurrent (Lemma~\ref{L:ContinuityAtZero}).

\subsection*{The paper is organized as follows:} In Section~\ref{sec:Preliminaries}, we define relative free factor complex and relative outer space. We also recall some basics about train track maps. In Section~\ref{sec:NorthSouth}, we prove Theorem C. We define a relative Whitehead graph to prove convergence for simplicial trees. We also define a transverse covering and discuss diagonal leaves to carry out the Levitt and Lustig $\mathcal{Q}$ map proof for convergence of trees with dense orbits. We discuss the intersection form in Section~\ref{sec:IntersectionForm} and conclude by giving a proof of Theorem A in Section~\ref{sec:MainProof}.

\subsection*{Acknowledgments} I am grateful to my advisor Mladen Bestvina for his guidance, support and endless patience throughout this project. I am also grateful to Camille Horbez for his invaluable help with the intersection form and many fruitful discussions about dual laminations of trees. I would like to express my gratitude to Arnaud Hilion for explaining his unpublished work \cite{CHL:LaminationLimit} (see Lemma~\ref{L:LaminationInclusion} and  Proposition~\ref{P:LaminationLimit}). I would also like to thank Mark Feighn for helpful comments on an earlier draft and Derrick Wigglesworth for many helpful conversations. I gratefully acknowledge the Mathematical Sciences Research Institute, Berkeley, California, where the final stage of the present work was completed during Fall 2016. I would also like to thank the referee for helpful suggestions.      
\section{Preliminaries}\label{sec:Preliminaries}
\subsection{Outer space} Culler Vogtmann's outer space (resp. unprojetivized outer space), $CV_n$ (resp. $cv_n$), is defined in \cite{CV:OuterSpace} as the space of $\free$-equivariant homothety (resp. isometry) classes of minimal, free and simplicial action of $\free$ by isometries on metric simplicial trees with no vertices of valence two. 

An $\free$-tree is an $\mathbb{R}$-tree with an isometric action of $\free$. An $\free$-tree is called \emph{very small} if the action is minimal, arc stabilizers are either trivial or maximal cyclic and tripod stabilizers are trivial. Outer space can be embedded into $\mathbb{R}^{\free}$ via translation lengths of elements of $\free$ in a tree in $cv_n$ \cite{CM:LengthFunction}. The closure of $CV_n$ under the embedding into $\mathbb{PR}^{\free}$ was identified in \cite{BF:OuterLimits} and \cite{CL:BoundaryOuterSpace} with the space of all very small $\free$-trees. We denote by $\overline{CV}_n$ the closure of outer space and by $\partial CV_n$ its boundary. 
\subsection{Marked graphs and topological representatives} We recall some basic definitions from \cite{BH:TrainTracks}. Identify $\free$ with $\pi_1(\mathcal{R}, \ast)$ where $\mathcal{R}$ is a rose with $n$ petals and $n$ is the rank of $\free$. A \emph{marked graph} $G$ is a graph of rank $n$, all of whose vertices have valence at least two, equipped with a homotopy equivalence $m : \mathcal{R}\to G$ called a \emph{marking}. The marking determines an identification of $\free$ with $\pi_1(G,m(\ast))$. 

A homotopy equivalence $\phi: G \to G$ induces an outer automorphism of $\pi_1(G)$ and hence an element $\oo$ of $\op{Out}(\free)$. If $\phi$ sends vertices to vertices and the restriction of $\phi$ to edges is an immersion then we say that $\phi$ is a \emph{topological representative} of $\oo$. 

A \emph{filtration} for a topological representative $\phi: G \to G$ is an increasing sequence of (not necessarily connected) $\phi$-invariant subgraphs $\fltr$. The closure of $G_r \setminus G_{r-1}$, denoted $H_r$, is a subgraph called the \emph{$r^{th}$-stratum}. Let $\gamma$ be a reduced path in $G$. Then  $\phi(\gamma)$ is the image of $\gamma$ under the map $\phi$. We will denote the tightened image of $\phi(\gamma)$ by $[\phi(\gamma)]$.  
 

\subsection{Relative train track map} We recall some more definitions from \cite{BH:TrainTracks}. A \emph{turn} in a marked graph $G$ is a pair of oriented edges of $G$ originating at a common vertex. A turn is non-degenerate if the edges are distinct, it is degenerate otherwise. 


We associate a matrix called \emph{transition matrix}, denoted $M_r$, to each stratum $H_r$. The $ij^{th}$ entry of $M_r$ is the number of occurrences of the $i^{th}$ edge of $H_r$ in either direction in the image of the $j^{th}$ edge under $\phi$. A non-negative matrix $M$ is called  \emph{irreducible} if for every $i, j$ there exists $k(i,j)>0$ such that the $ij^{th}$ entry of $M^k$ is positive. A matrix is called \emph{primitive} or \emph{aperiodic} if there exists $k>0$ such that $M^k$ is positive. A stratum is called \emph{zero stratum} if the transition matrix is the zero matrix. If $M_r$ is irreducible then its Perron-Frobenius eigenvalue $\lambda_r$ is greater than equal to 1. We say a stratum with an irreducible transition matrix is \emph{exponentially growing (EG)} if $\lambda_r >1$, it is called \emph{non-exponentially growing (NEG)} otherwise. 

A topological representative $\phi : G \to G$ of a free group outer automorphism $\oo$ is a \emph{relative train track map} with respect to a filtration $\fltr$ if $G$ has no valence one vertices, if each non-zero stratum has an irreducible matrix and if each exponentially growing stratum satisfies the following conditions: 
\begin{itemize}
 \item If $E$ is an edge in $H_r$, then the first and the last edges in $[\phi(E)]$ are also in $H_r$. 
\item If $\gamma \in G_{r-1}$ is a non-trivial path with endpoints in $H_r \cap G_{r-1}$, then $[\phi(\gamma)]$ is a non-trivial path with endpoints in $H_r \cap G_{r-1}$. 
\item For each $r$-legal path $\beta \subset H_r$, $[\phi(\beta)]$ is $r$-legal.  
\end{itemize}

\subsection{BFH Laminations }\label{subsec:Laminations}
In \cite{BFH:Tits}, Bestvina, Feighn and Handel defined a dynamic invariant called the attracting lamination associated to an EG stratum of a relative train track map $\phi: G \to G$. The elements of the lamination are called \emph{leaves}. 


Let $\mathcal{B}$ be the space of lines defined as the quotient of $\partial^2 \free := (\partial \free \times \partial \free - \Delta)/ \mathbb{Z}_2$ by the action of $\free$, where $\Delta$ denotes the diagonal. We say $\beta^{\p} \in \mathcal{B}$ is weakly attracted to $\beta \in \mathcal{B}$ under the action of $\oo$ if $[\oo^k(\beta^{\p})]$ converges to $\beta$. A subset $U \subset \mathcal{B}$ is an \emph{attracting neighborhood} of $\beta$ for the action of $\oo$ if $[\oo(U)]$ is a subset of $U$ and if $\{[\oo^k(U)] : k \geq 0\}$ is a neighborhood basis for $\beta$ in $\mathcal{B}$. A bi-infinite path $\sigma$ in a marked graph is \emph{birecurrent} if every finite subpath of $\sigma$ occurs infinitely often as an unoriented subpath of each end of $\sigma$. An element of $\mathcal{B}$ is birecurrent if some realization in a marked graph is birecurrent. 

A closed subset $\Lambda^+$ of $\mathcal{B}$ is called an \emph{attracting lamination} for a free group outer automorphism $\oo$ if it is the closure of a line $\beta$ that is bireccurent, has an attracting neighborhood for the action of some iterate of $\oo$ and is not carried by a $\oo$-periodic free factor of rank one. The line $\beta$ is said to be a \emph{generic leaf} of $\Lambda^+$. In this paper, we will look at the lift of the attracting lamination to $\partial^2 \free$ and denote it also by $\Lambda^+$. 

\begin{lemma}[{\cite[Lemma 3.1.9]{BFH:Tits}}]\label{L:BFH lamination}
Suppose that $\phi : G \to G$ is a relative train track map with respect to a filtration $\fltr$ representing $\oo$ and $H_r$ is an aperiodic EG stratum. Then there is an attracting lamination $\Lambda^+_r$ with generic leaf $\beta$ so that $H_r$ is the highest stratum crossed by a realization of $\beta$ in $G$. \end{lemma}

\subsection{CHL Laminations}
In \cite{CHL:AlgebraicLaminationsI} Coulbois, Hilion and Lustig defined three laminations associated to $\free$: algebraic laminations, symbolic laminations and laminary languages. They established the equivalence of the three approaches. An \emph{algebraic lamination} is a non-empty, closed and $\free$-invariant subset of $\partial^2 \free$. 
Let $\Lambda^2(\free)$ be the (compact, metric) space of algebraic lamination in $\free$. 

\begin{definition}[Convergence of laminations {\cite[Remark6.3]{CHL:AlgebraicLaminationsI}}]\label{D:ConvergenceLaminations}
We say a sequence of algebraic laminations $L_n$ converges to a lamination $L_{\infty}$ in $\Lambda^2(\free)$ if the following holds: let $L_n^s$ and $L_{\infty}^s$ be the symbolic laminations associated to $L_n$ and $L_{\infty}$ respectively with respect to some (any) basis of $\free$. Given a symbolic lamination $L^s$, let $\mathcal{L}_m(L^s)$ be the set of words in $L^s$ of length less than equal to $m$. 
The sequence $L_n$ converges to $L_{\infty}$ if for every $m\geq 1$ there exists a $K(m) \geq 1$ such that for every $k \geq K(m)$, $\mathcal{L}_m(L_k^s) = \mathcal{L}_m(L_{\infty}^s)$. 
\end{definition}
\subsection{Dual lamination of an $\mathbb{R}$-tree } Associated to $T\in \overline{CV_n}$ is a \emph{dual algebraic lamination} $L(T)$, which is defined as follows in \cite{CHL:DualLaminationsII}: let $$L_{\epsilon}(T) := \overline{\{(g^{-\infty}, g^{\infty})| l_T(g) < \epsilon, g \in \free \}},$$
so $L_{\epsilon}(T)$ is an algebraic lamination and set $\displaystyle{L(T) := \bigcap_{\epsilon>0}L_{\epsilon}(T)}$.

For trees in $CV_n$, $L(T)$ is empty. For another example of a dual algebraic lamination, consider an atoroidal fully irreducible outer automorphism $\Psi$ and its unstable tree $T^-_{\Psi}$. The unstable tree is the limit of the sequence $T.\Psi^{-n}$ in $\overline{CV}_n$ for any free simplicial tree $T$ in $CV_n$ (see \cite{BFH:Laminations} for detailed definition). By a result of \cite{KL:InvariantLaminations}, if $\Lambda^+_{\Psi}$ is the attracting lamination associated to $\Psi$ (as given by Lemma~\ref{L:BFH lamination}), then $L(T^-_{\Psi})$ is the diagonal closure of $\Lambda^+_{\Psi}$, that is, if $(X, X^{\p}) \in \partial^2 \free$ and $(X, X^{\p\p}) \in \partial^2 \free$ are in $\Lambda^+_{\Psi}$, which is a subset of $L(T^-_{\Psi})$, and $X^{\p} \neq X^{\p\p}$ then $(X^{\p}, X^{\p\p})$ is also in $L(T^-_{\Psi})$. 

For trees in $\partial CV_n$ with dense orbits there are two more definitions given in \cite{CHL:DualLaminationsII}: 

\begin{itemize}
\item $L_{\infty}(T)$ :  For a basis $\mathfrak{B}$ of $\free$, let $L^1_{\mathfrak{B}}(T) \subset \partial \free$ be the set of one sided infinite words with respect to $\mathfrak{B}$ that are bounded in $T$. By \cite[Proposition 5.2]{CHL:DualLaminationsII} this set is independent of the basis and henceforth will be denoted $L^1(T)$. The lamination $L_{\infty}(T)$ is the algebraic lamination defined by the recurrent laminary language in $\mathfrak{B}^{\pm}$ associated to $L^1(T)$. It is shown in the same paper that this definition is also independent of the basis.

\item $L_{\mathcal{Q}}(T)$ : See Definition~\ref{D:QLamination}.
\end{itemize}
The equivalence of the three definitions of dual lamination of a tree in $\partial CV_n$ with dense orbits is established in the same paper. Note that $L_{\infty}(T)$ can also be defined for trees which don't have dense orbits but it might not be equal to $L(T)$. 
 

\subsection{Free factor system}
A free factor system of $\free$ is a finite collection of conjugacy classes of proper free factors of $\free$ of the form $\ffa = \{[A_1], \ldots, [A_k] \}$, where $k\geq 0$ and $[\cdot]$ denotes the conjugacy class of a subgroup, such that there exists a free factorization $\free = A_1 \ast \cdots \ast A_k \ast F_N$. We refer to the free factor $F_N$ as the \emph{cofactor} of $\ffa$ keeping in mind that it is not unique, even up to conjugacy. There is a partial ordering $\sqsubset$ on the set of free factor systems given as follows: $\ffa \sqsubset \ffa^{\prime}$ if for every $[A_i] \in \ffa$ there exists $[A_j^{\prime}] \in \ffa^{\prime}$ such that $A_i \subset A_j^{\prime}$ up to conjugation. The free factor systems $\emptyset$ and $\{[\free]\}$ are called \emph{trivial free factor systems}. We define \emph{rank}$\mathbf{(\ffa)}$ to be the sum of the ranks of the free factors in $\ffa$. Let $\zeta(\ffa) = k+N$. 
   
The main geometric example of a free factor system is as follows: suppose $G$ is a marked graph and $K$ is a subgraph whose non-contractible connected components are denoted $C_1, \ldots, C_k$. Let $[A_i]$ be the conjugacy class of a free factor of $\free$ determined by $\pi_1(C_i)$. Then $\ffa =\{ [A_1], \ldots, [A_k]\}$ is a free factor system. We say $\ffa$ is \emph{realized by $K$} and we denote it by $\mathcal{F}(K)$. 
 
\subsection{Relative free factor complex}\label{subsec:Relative free factor complex}
Let $\ffa$ be a non-trivial free factor system of $\free$. In \cite{HM:RelativeComplex} the complex of free factor systems of $\free$ relative to $\ffa$, denoted $\mathcal{FF}(\free; \ffa)$, is defined to be the geometric realization of the partial ordering $\sqsubset$ restricted to the set of non-trivial free factor systems $\dd$ of $\free$ such that $\ffa \sqsubset \dd$ and $\dd \neq \ffa$. The \emph{exceptional} free factor systems are certain ones for which $\rfa$ is either empty or zero-dimensional. They can be enumerated as follows: 
\begin{itemize}
 \item $\ffa = \{ [A_1], [A_2] \}$ with $\free = A_1 \ast A_2$. In this case $\rfa$ is empty. 
\item $\ffa = \{[A]\}$ with $\free = A \ast \mathbb{Z}$. In this case $\rfa$ is $0$-dimensional. 
\item $\ffa = \{ [A_1], [A_2], [A_3]\}$ with $\free =A_1 \ast A_2\ast A_3$. In this case $\rfa$ is also $0$-dimensional. 
\end{itemize}

\begin{theorem}[{\cite{HM:RelativeComplex}}]
For any non-exceptional free factor system $\ffa$ of $\free$, the complex $\rfa$ is positive dimensional, connected and hyperbolic. 
\end{theorem}
\subsection{Fully irreducible relative to $\ffa$}
An outer automorphism $\oo \in \relOut$ is called \emph{irreducible relative to} $\mathbf{\ffa}$ if there is no $\oo$-invariant non-trivial free factor system that properly contains $\ffa$. If every power of $\oo$ is irreducible relative to $\ffa$ then we say that $\oo$ is \emph{fully irreducible relative to} $\mathbf{\ffa}$ (or relatively fully irreducible).  

Let $\oo \in \relOut$. Then by \cite[Lemma 2.6.7]{BFH:Tits} there exists a relative train track map for $\oo$ denoted $\phi : G \to G$ and filtration $\Rfltr$ such that $\ffa = \mathcal{F}(G_s)$ for some filtration element $G_s$. If $\oo$ is fully irreducible relative to $\ffa$ then $\ffa = \mathcal{F}(G_{r-1})$ and the top stratum $H_r$ is an EG stratum with Perron-Frobenius eigenvalue $\PFevalue>1$. 

For $\oo$, a fully irreducible outer automorphism relative to $\ffa$, let $\lamination$ be the attracting lamination associated to the top stratum $H_r$. We will denote by $\lamination(G)$ the realization of $\lamination$ in the graph $G$. 
\subsection{Relative outer space}\label{subsec:RelativeOuterSpace}
In \cite{GL:RelativeOuterSpace}, Guirardel and Levitt define relative outer space for a countable group that splits as a free product $$G = G_1 \ast \ldots \ast G_k \ast F_N$$ where $N+k \geq 2$. In \cite{H:HorbezBoundary}, Horbez shows that the closure of relative outer space is compact and characterizes the trees in the closure of relative outer space. 
 
In our setting $G = \free$ and it splits as $\free = A_1 \ast \ldots \ast A_k \ast F_N$ for $k\geq 0$. Let $\ffa = \{[A_1], \ldots, [A_k]\}$ be the associated free factor system of $\free$. The group of automorphisms associated to such a decomposition is $\relOut$ consisting of those outer automorphisms that preserve the conjugacy class of each $A_i$. 

Subgroups of $\free$ that are conjugate into a free factor in $\ffa$ are called \emph{peripheral} subgroups. An $(\free, \ffa)$-tree is an $\mathbb{R}$-tree with an isometric action of $\free$, in which every peripheral subgroup fixes a unique point. A $\emph{Grushko}$ $(\free,\ffa)$-tree is a minimal, simplicial metric $(\free, \ffa)$-tree whose set of point stabilizers is exactly the free factor system $\ffa$ and edge stabilizers are trivial. Two $(\free, \ffa)$-trees are equivalent if there exists an $\free$-equivariant isometry between them. An $(\free,\ffa)$-tree $T$ is \emph{small} if arc stabilizers in $T$ are either trivial, or cyclic and non-peripheral. A small $(\free, \ffa)$-tree $T$ is \emph{very small} if in addition the non-trivial arc stabilizers in $T$ are closed under taking roots and tripod stabilizers are trivial. 

The \emph{unprojectivized relative outer space} $\mathcal{O}(\free, \ffa)$ is the space of all equivalence classes of Grushko $(\free, \ffa)$-trees. \emph{Relative outer space}, denoted $\relCV$, is the space of homothety classes of trees in $\mathcal{O}(\free, \ffa)$. 

\begin{example}
\begin{enumerate}[(a)]
\item Let $\free = A_1 \ast A_2$. In this case relative outer space is just a point represented by a one edge splitting with vertex stabilizers $A_1$ and $A_2$ and trivial edge stabilizer. 
\item Let $\free = A_1 \ast \mathbb{Z}$. In this case relative outer space is one dimensional. A schematic is shown in Figure~\ref{F:RelativeOuterSpaceEx}(i). The central vertex $v$ in (i) corresponds to the graph shown in (ii) and the end points of the one simplices in (i) correspond to graphs shown in (iii). 
\begin{figure}[h]
\centering
\includegraphics[scale = .6]{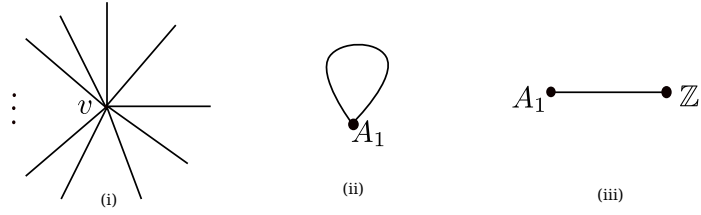} 
\caption{Relative Outer Space} \label{F:RelativeOuterSpaceEx}
\end{figure}

\item Let $\free = A_1 \ast A_2 \ast A_3$. In this case relative outer space is unbounded with respect to the simplicial metric. 
\end{enumerate}
\end{example}
The graph of groups decomposition of $\free$ represented in Figure~\ref{F:RelativeRose} is called a \emph{relative rose}.  
\begin{figure}[h]
\centering
\includegraphics[scale = .5]{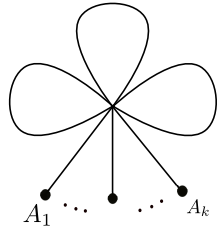} 
\caption{Relative Rose} \label{F:RelativeRose}
\end{figure}
\subsection{Boundary of $\free$}\label{subsec:Boundary}
Given $\free$ and a fixed basis $\mathfrak{B}$ of $\free$, let $\op{Cay}(\free, \mathfrak{B})$ be the Cayley graph of $\free$ with respect to $\mathfrak{B}$. The space of ends of the Cayley graph is called the \emph{boundary of} $\free$, denoted by $\partial \free$. Let $\Delta$ denote the diagonal in $\partial \free \times \partial \free$. Let $\partial^2 \free := (\partial \free \times \partial \free - \Delta) / \mathbb{Z}_2$ be the \emph{space of unoriented, bi-infinite geodesics} in $\op{Cay}(\free, \mathfrak{B})$. Finite paths $\gamma$ in $\op{Cay}(\free, \mathfrak{B})$ determine two-sided cylinder sets, denoted $C(\gamma)$, which form a basis for the topology of $\partial^2\free$. Compact open sets are given by finite disjoint union of cylinder sets. 

Let $\RelativeBasis$ be a basis of $\free$ such that a basis of $\ffa$ is a subset of $\RelativeBasis$. Specifically, $$\RelativeBasis = \{a_{11}, \ldots a_{11_s}, \ldots , a_{i1}, \ldots, a_{ii_s}, \ldots, a_{k1}, \ldots, a_{kk_s}, b_1, \ldots, b_p\}$$ where $a_{ij} \in A_i$ and $b_i \notin A$ for any $[A] \in \ffa$. 
If $\op{rank}(\ffa) = \op{rank}(\free)$ then $p=0$. We call such a basis a \emph{relative basis} of $\free$. 

Given a free factor $A$ say a one-sided infinite geodesic starting at the base point in $\op{Cay}(\free, \RelativeBasis)$ is in $\partial A$ if eventually it crosses only edges labeled by words in $A$. Note that $\partial A$ is an $\free$-equivariant set. Define $\partial \ffa := \bigsqcup_{i=1}^k \partial A_i$ and let $\partial^2 A$ be the closure of the set of bi-infinite geodesics in $\partial^2 \free$ which are lifts of conjugacy classes in $A$. Let $\partial^2 \ffa := \bigsqcup_{i=1}^k \partial^2 A_i$.      

Let $[\free \setminus \ffa]$ be the set of all conjugacy classes in $[\free]$ that are not contained in a free factor in $\ffa$ and let $\free \setminus \ffa$ be the set of all words in $\free$ that are not contained in a free factor in $\ffa$. Note that an element of $\free \setminus \ffa$ can be contained in the free product of distinct free factors of $\ffa$.  

\subsection{Relative Currents}\label{subsec:RelativeCurrents}
In \cite{G:RelativeCurrents}, we relativized the notion of measured currents as defined in \cite{Martin}. Let $\ffa$ be a non-trivial free factor system such that $\zeta(\ffa)\geq 3$. We define $\mathbf{Y} := \partial^2 \free \setminus \partial^2 \ffa$ with the subspace topology from $\partial^2\free$. 
Let $\mathcal{C}(\mathbf{Y})$ be the collection of compact open sets in $\mathbf{Y}$. 

A \emph{relative current} is an additive, non-negative, $\free$-invariant and flip-invariant function on $\mathcal{C}(\mathbf{Y})$. A relative current is uniquely determined by its values on cylinder sets in $\mathcal{C}(\mathbf{Y})$ determined by finite paths corresponding to words in $\free \setminus \ffa$.  

\begin{example}[Relative current]
Consider a conjugacy class $\alpha \in [\free \setminus \ffa]$ such that $\alpha$ is not a power of any other conjugacy class in $[\free]$ and let $w$ be a word in $\free \setminus \ffa$. Let $C(w)$ be the cylinder set in $\mathbf{Y}$ corresponding to a path determined by $w$ starting at the base point of the Cayley graph. Then $\eta_{\alpha}(w):= \eta_{\alpha}(C(w))$ is the number of occurrences of $w$ in the cyclic words $\alpha$ and $\overline{\alpha}$. Equivalently, we can also count the number of lifts of $\alpha$ that cross the path determined by $w$ starting at the base point in the Cayley graph. Such currents and their multiples are called \emph{rational relative currents}. For example, let $\free = \la a, b\ra, \ffa = \{[\la a \ra]\}$ and let $\alpha = abaab$. Then $ \eta_{\alpha}(b) = 2, \eta_{\alpha}(ba) = 2, \eta_{\alpha}(abab) = 1$. \end{example}

Let $\rc$ denote the space of relative currents. A subbasis for the topology of $\rc$ is given by the sets $\{\eta \in \rc: |\eta(C) - \eta_0(C)|\leq \epsilon \}$ where $\eta_0 \in \rc$, $C \in \mathcal{C}(\mathbf{Y})$ and $\epsilon >0$. Let $\prc$ be the space of projectivized relative currents, which is shown to be compact in \cite{G:RelativeCurrents}. Let $\mrc$ be the closure in $\prc$ of the relative currents corresponding to conjugacy classes in $[\free \setminus \ffa]$ which are contained in a non-trivial free factor system containing $\ffa$. The main result about relative currents is the following:

\begin{thmB}[{\cite{G:RelativeCurrents}}]
Let $\ffa$ be a non-trivial free factor system of $\free$ such that $\zeta(\ffa) \geq 3$. Let $\oo \in \relOut$ be fully irreducible relative to $\ffa$. Then $\oo$ acts on $\mrc$ with uniform north-south dynamics: there are only two fixed points $\StableCurrent$ and $\UnstableCurrent$ and 
for every compact set $K$ of $\mrc$ that does not contain $\UnstableCurrent$ (rep. $\StableCurrent$), and for every neighborhood $U$ (resp. $V$) of $\StableCurrent$ (resp. $\UnstableCurrent$), there exists $N \geq 1$ such that for all $n \geq N$ we have $\oo^n(K) \subseteq U$ (resp. $\oo^{-n}(K) \subseteq V$).    
\end{thmB}
\section{North-south dynamics on the closure of relative outer space}\label{sec:NorthSouth}
Our method to prove Theorem C is a generalization of the proof by Levitt and Lustig \cite{LL:NSDynamics} to show that a fully irreducible automorphism acts with uniform north-south dynamics on the closure of outer space. Let $\oo$ be a fully irreducible automorphism relative to $\ffa$. 

\begin{notation}\label{N:maps} Let $\phi_0^{\p} : G^{\p} \to G^{\p}$ be a relative train track representative of $\oo$, where $G^{\p}$ is a marked metric graph in $CV_n$, with filtration $\emptyset = G_0 \subset G_1 \subset \ldots \subset G_r = G^{\p}$ such that $\ffa = \mathcal{F}(G_{r-1})$ and the top stratum $H_r$ is an EG stratum with Perron-Frobenius eigenvalue $\mathbf{\PFevalue} >1$. We denote by $\lamination$ the attracting lamination associated to $H_r$ and by $\lamination(G^{\p})$ its realization in $G^{\p}$. Let $T_{G^{\p}}$ be the universal cover of $G^{\p}$ and let $\phi^{\p}: \AbsTree \to \AbsTree$ be a lift of the map $\phi_0^{\p}: G^{\p} \to G^{\p}$ which satisfies $\oo(g) \circ \phi^{\p} = \phi^{\p} \circ g$ for $g \in \free$.  
\end{notation}

\begin{definition}[$\ffa$-train track map]\label{D: A TrainTrack}Let $\RelTTTree$ be the tree in $\RelCV$ obtained by equivariantly collapsing the maximal $\phi^{\p}$-invariant proper forest of $\AbsTree$. We denote the collapse map by $\pi: T_{G^{\p}} \to \RelTTTree$. See Figure~\ref{F:Collapse}. The map $\phi^{\p}: T_{G^{\p}} \to T_{G^{\p}}$ descends to a map $\phi: \RelTTTree \to \RelTTTree$ representing $\oo$. Let $G = T_G / \free$ and $\phi_0: G \to G$ be the corresponding map. We say $\phi_0$ is an \emph{$\ffa$-train track} representative of $\oo$. \end{definition}

\begin{figure}[h]
\centering
\includegraphics[scale=.5]{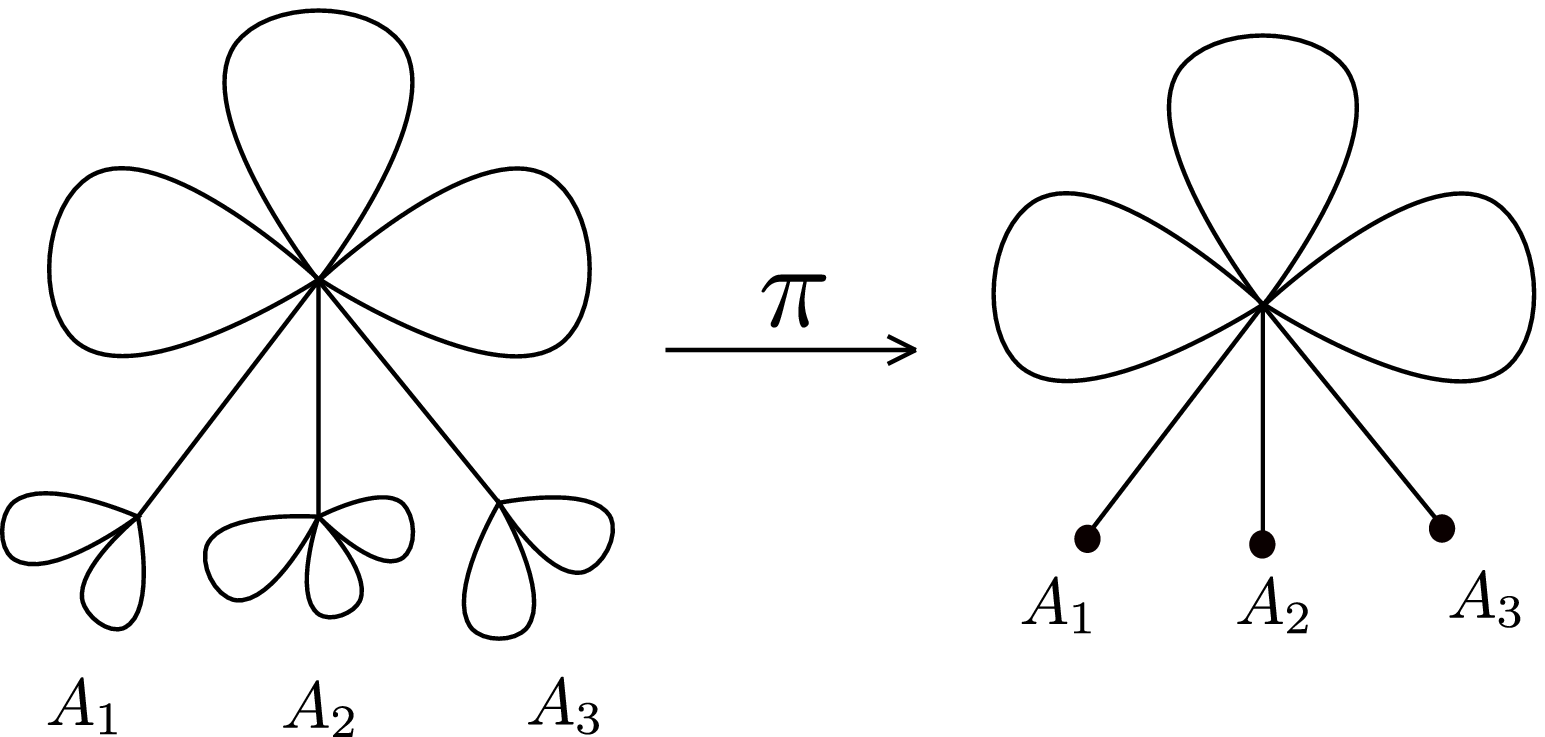}
\caption{}\label{F:Collapse}
\end{figure}

\subsection{Stable and unstable trees}
$\relOut$ acts on $\RelCVClo$ via $$ l_{T\Psi}(\alpha) = l_T(\Psi(\alpha))$$ for $\Psi \in \relOut$ and for every conjugacy class $\alpha \in [\free]$, where $l_T(\alpha)$ is the translation length of $\alpha$ in $T$. A stable tree $\StableTreePhi$ of $\oo$ is defined as follows: 
$$ \StableTreePhi = \lim_{p \to \infty} \frac{\RelTTTree \phi^p}{\PFevalue^p}.$$
In other words, $$ l_{\StableTreePhi}(\alpha) = \lim_{p \to \infty}\frac{l_{\RelTTTree}(\phi^p\alpha)}{\PFevalue^p}.$$

  
The stable tree is well defined projectively and hence we denote the projective class by $\StableTree$. The unstable tree, denoted $\UnstableTree$, of $\oo$ is defined to be the stable tree of $\oo^{-1}$. The fact that $T_{\oo}^{\pm}$ do not depend on the choice of the $\ffa$-train track map $\phi$ follows from the same arguments as in \cite[Lemma 3.4]{BFH:Laminations} whose relative version is stated below. 

\begin{proposition}\label{P:Proof Strategy}
Let $T \in \relCVClo$. Suppose there exists a tree $T_0 \in \relCV$, an equivariant map $h : T_0 \to T$, and a bi-infinite geodesic $\gamma_0 \subset T_0$ representing a generic leaf $\gamma$ of $\lamination$ such that $h(\gamma_0)$ has diameter greater than 2BCC($h$). Then
 \begin{enumerate}[(a)]
 \item $h(\gamma_0)$ has infinite diameter in $T$.  
 \item there exists a neighborhood $V$ of $T$ such that $\oo^p(V)$ converges to $\StableTree$ uniformly as $p \to \infty$. 
\end{enumerate}\end{proposition}
The proof of Proposition~\ref{P:Proof Strategy} is essentially the same as in the absolute case in \cite[Lemma 3.4]{BFH:Laminations} and \cite[Proposition 6.1]{LL:NSDynamics}. For completeness, we give a proof here. 

\begin{proof}[Proposition~\ref{P:Proof Strategy} (a)]
Fix an equivariant map $\mu : \RelTTTree \to T_0$ with some bounded backtracking. Let $\gamma_0$ be the tightened image of $\gamma$, a generic leaf of $\lamination$, under $\mu$. Let $h: T_0 \to T$ be the $\free$-equivariant map as given in the proposition. If $AB \subset \RelTTTree$ is a segment, denote by $l_T(\nu(AB))$ the length of the tightened image of $AB$ under $\nu = h \circ \mu$. Let $\op{Lip}(\nu)$ be the Lipschitz constant of $\nu$ and let $\op{BCC}(\nu)$ be the bounded backtracking constant. We have $\op{BCC}(\nu) \leq \op{Lip}(\mu)\op{BCC}(\mu)+ \op{Lip}(h)\op{BCC}(h)$.

By assumption, there is a segment $A_0B_0$ in $\gamma_0$ such that its image in $T$ by $h$ has length greater than $2\op{BCC}(h)$. Let $\sigma$ be the central subsegment of $h(A_0)h(B_0)$ whose length is $l_T(h(A_0)h(B_0))-2\op{BCC}(h)$. We can find a segment $AB \subset \gamma$ such that its image by $\mu$ contains $A_0 B_0$ and hence its tightened image by $\nu$ contains $\sigma$. Choose $m_0$ such that $\phi^{m_0}(e)$ contains a translate of $AB$ for every edge $e$ in $\RelTTTree$. 
If $\beta$ is any leaf segment contained in $\lamination$, then $l_T(\nu(\phi^{m_0}(\beta))) \geq l_T(\sigma) |\beta|$ where $|\beta|$ is the simplicial length of $\beta$ in $\RelTTTree$. Since $\phi^{m_0}(\beta)$ is also a subsegment of $\gamma_0$ if $\beta$ is, we conclude that $h(\gamma_0)$ has infinite diameter in $T$.  

\end{proof}

\begin{proof}[Proposition~\ref{P:Proof Strategy}(b)]
Since $h(\gamma_0)$ has infinite diameter in $T$, for every edge $e \in \RelTTTree$, the length $l_T(\nu(\phi^p(e)))$ tends to infinity with $p$. Let $\beta$ be an arbitrary edge path in $\RelTTTree$ and let $$d_+(\beta) := \lim_{p \to \infty} \frac{l_{\RelTTTree}(\phi^p(\beta))}{\PFevalue^p}.$$ 

We claim that for an arbitrary edge path $\beta$ in $\RelTTTree$, $$ \lim_{p \to \infty} \frac{l_T(\nu(\phi^p(\beta)))}{\PFevalue^p d_+(\beta)} = c$$ and the convergence is uniform (depending only on $Lip(\nu)$), that is, it is independent of $\beta$. Indeed, if $\beta$ is a subsegment of a generic leaf of $\lamination(\RelTTTree)$, then the claim can be proved using Perron-Frobenius theorem applied to $\phi$. If $\beta$ is not a leaf subsegment then it can be written as a concatenation of finitely many leaf segments and then by using bounded cancellation the claim can be proved. See proof of \cite[Lemma 7.1, 7.2]{LL:NSDynamics} for details. 

We now give a proof of Proposition~\ref{P:Proof Strategy}(b). Let $g \in \free$ be a nonperipheral conjugacy class. For $n\geq 1$, let $\beta_n$ be a fundamental domain for the action of $g^n \in \free$ on $\RelTTTree$. Let $||g||_T$ be the translation length of $g$ in $T$. Since $l_T(\nu(\phi^p(\beta_n))) - 2\op{BCC}(\nu) \leq ||\oo^p (g^n)||_T = ||g^n||_{T\phi^p} \leq l_T(\nu(\phi^p(\beta_n)))$ and $d_+(\beta_n)=||g^n||_{\StableTree}$, by the claim, we get 
$$ \frac{||g^n||_{T\phi^p}}{c \PFevalue^p} \to ||g^n||_{\StableTree} \text{ as } p \to \infty.$$ Since $||g||_T = \lim_{n \to \infty} ||g^n||_T / n$, we get that $T$ converges to $\StableTree$ under forward iteration by $\oo$. 

For $T^{\p}$ close to $T$, there exists $h^{\p}: T_0 \to T^{\p}$, linear on edges such that images of edges have approximatey the same length in $T^{\p}$ as in $T$. Thus $\op{Lip}(h)$ is close to $\op{Lip}(h^{\p})$ and thus $\op{Lip}(\nu^{\p})$ is close to $\op{Lip}(\nu)$. Since the convergence in the claim depends only on the lipschitz constant of $\nu$, we can find a small neighborhood $V$ of $T$ where the convergence is uniform.

\end{proof}

Our goal now is to prove that every tree $T \in \relCVClo$ satisfies the assumptions of Proposition~\ref{P:Proof Strategy} if we allow $\gamma$ to be either in $\lamination$ or $\Rlamination$. We prepare ourselves for this task by proving some results about Whitehead graphs, transverse coverings and $\mathcal{Q}$ map in the next three sections which will then be put together in Section~\ref{subsec:NS for trees proof} to complete the proof of Theorem C. 
\subsection{Relative Whitehead graph}\label{subsec:Relative Wh graph} The main lemma in this section is Lemma~\ref{L:WhGraph} which is used in the proof of Lemma~\ref{L:Tcases}. We first recollect some observations in the absolute case about the Whitehead graph for a fully irreducible automorphism. We then define a relative Whitehead graph and make similar observations for a fully irreducible automorphism relative to $\ffa$. 

Let $\psi: \Gamma \to \Gamma$ be a train track representative of a fully irreducible automorphism where $\Gamma \in CV_n$ and let $\Lambda^+_{\psi}$ be the attracting lamination. 

\begin{definition}[Whitehead graph \cite{BFH:Laminations}]\label{D:WhGraph} At a vertex $v$ of $\Gamma$ the \emph{Whitehead graph}, denoted $\op{Wh}(v)$, is defined as follows: the vertices are given by the outgoing edges incident at $v$ and two vertices are joined by an edge if the corresponding outgoing edges in $\Gamma$ form a \emph{$\Lambda_{\psi}^+$-legal turn}, that is, there is a $\psi$-iterate of an edge of $\Gamma$ that crosses that turn. 
\end{definition} 
 
If $\psi(v) = w$ where $v,w$ are vertices in $\Gamma$ then $\psi$ induces a simplicial map from $\op{Wh}(v)$ to $\op{Wh}(w)$. 

\begin{definition}[\cite{BFH:Laminations}]
A finitely generated subgroup $H$ of $\free$ \emph{carries} a lamination $\Lambda$ if there exists a marked metric graph $\Gamma_0$, an isometric immersion $i: \Gamma_H \to \Gamma_0$ with $\pi_1(i(\Gamma_H)) = H$ and an isometric immersion $l: \mathbb{R} \to \Gamma_H$ such that $i \circ l$ is a generic leaf of $\Lambda(\Gamma_0)$. 
\end{definition}

\begin{proposition}[{\cite[Lemma 2.1, Proposition 2.4]{BFH:Laminations}}]\label{P:Absolute Wh Graph}
\begin{enumerate}[(a)]
\item At every vertex of $\Gamma$ the Whitehead graph is connected. 
\item Suppose $\pi: \Gamma^{\prime} \to \Gamma$ is a finite sheeted covering space and $\psi^{\prime}: \Gamma^{\prime} \to \Gamma^{\prime}$ is a lift of $\psi$. Then the transition matrix of $\psi^{\prime}$ is primitive and the Whitehead graph of $\psi^{\prime}$ at a vertex $v$ of $\Gamma^{\prime}$ is the lift of the Whitehead graph of $\psi$ at $\pi(v)$ and in particular is connected. 
\item If a finitely generated subgroup $H$ of $\free$ carries $\Lambda_{\psi}^+$ then $H$ is a finite index subgroup of $\free$. 
\end{enumerate}\end{proposition}

We now look at an example of the Whitehead graph of a fully irreducible automorphism relative to $\ffa$ to see why we need a notion of a relative Whitehead graph. 

\begin{example}\label{E:Disconnected Relative Wh Graph}
Let $F_4 = \la a, b, c, d \ra$, $\ffa = \{[\la a,b \ra]\}$ and $\oo$ a relative automorphism be given by 
$$\oo(a) = ab, 
\oo(b) = b,  
\oo(c) = cad, 
\oo(d) = dcad.$$
Let $\phi^{\p}_0: G^{\p} \to G^{\p}$ be a relative train track representative of $\oo$ where $G^{\p}$ is the rose on four petals labeled $a, b, c, d$ and vertex $v$. The Whitehead graph at $v$ is shown in Figure~\ref{F:WhGraphEx}: \begin{figure}[h]
\centering 
\includegraphics[scale=.6]{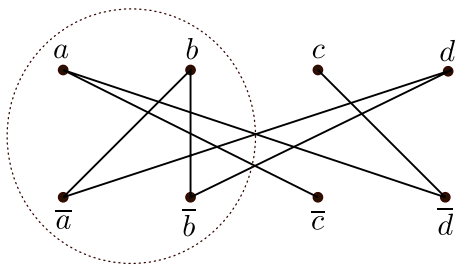}
\caption{Whitehead Graph for Example~\ref{E:Disconnected Relative Wh Graph}}
\label{F:WhGraphEx}
\end{figure}

The Whitehead graph at $v$ is disconnected with two gates $\{c, \overline{c}, a, \overline{d}\}$ and $\{\overline{a}, b, \overline{b}, d\}$. 
If we identify all the directions coming from the rose corresponding to $\la a, b \ra$ then we do get a connected graph. 
\end{example}

We will now define a relative Whitehead graph. Let $\phi_0': G' \to G'$ be the relative train track map and let $\phi_0: G \to G$ be the $\ffa$-train track representative of a relative fully irreducible automorphism $\oo$ from Definition~\ref{D: A TrainTrack}, with the attracting lamination $\lamination$. Recall from Section~\ref{subsec:Laminations}, that $\lamination$ is the closure of a birecurrent line $\beta$, which is obtained by looking at $\phi_0'$-iterates of an edge in the EG stratum of $G'$. 

\begin{definition}[Relative Whitehead graph]
 Let $v$ be a vertex of $G$ of valence greater than one. 
\begin{itemize}
\item If $v$ has trivial stabilizer then the relative Whitehead graph is defined as follows: the vertices are given by the outgoing edges incident at $v$ and two vertices are joined by an edge if the corresponding outgoing edges in $G$ form a \emph{$\lamination(G)$-legal turn}, that is, there is a $\phi_0$-iterate of an edge of $G$ that crosses that turn. 
\item If $v$ has a non-trivial stabilizer then do the following: consider the pre-images $v_1, \ldots, v_m$ of $v$ in $G'$ and build a graph as follows: 
the vertices are given by the outgoing edges incident at each $v_i$ in $G'$. Two vertices are joined by an edge if the corresponding outgoing edges in $G'$ form a $\lamination(G')$-turn, that is, there is a $\phi_0'$-iterate of an edge of $G'$ that crosses that turn.  

Now identify all the directions coming from the stratum $G_{r-1} \subset G'$, where $\mathcal{F}(G_{r-1}) = \ffa$. The graph thus obtained is called the relative Whitehead graph. The vertices of the relative Whitehead graph can also be thought of as the outgoing edges incident at $v$ and a vertex, denoted $v_A$, representing the non-trivial vertex stabilizer $A$. 
\end{itemize}  
\end{definition}

In Example~\ref{E:Disconnected Relative Wh Graph}, after collapsing the maximal invariant subgraph of $G^{\p}$ we get a graph $G$ which is a rose with two petals and vertex stabilizer $A = \la a, b \ra$. The relative Whitehead graph at the vertex of $G$ has vertices corresponding to $c, \overline{c}, d, \overline{d}, v_A$ and is shown in Figure~\ref{F:RelativeWhiteheadGraph}.
\begin{figure}[h]
\centering
\includegraphics[scale = .6]{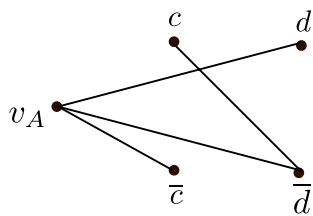}
\caption{Relative Whitehead graph for Example~\ref{E:Disconnected Relative Wh Graph}.}
\label{F:RelativeWhiteheadGraph}
\end{figure}

Before we state the next lemma, let's look at two examples of covering spaces for the relative rose, one by a finite index subgroup and another by an infinite index subgroup. Let $F_6 = \la a,b,c,d,e,f \ra$ and $\ffa  = \{ [\la a,b \ra], [\la c,d\ra]\}$.  
\begin{itemize}
\item Let $H = \la a, b, ef \ra$ be a subgroup of $\free$. The (infinite sheeted) cover of the relative rose corresponding to $H$ is shown in Figure~\ref{F:InfniteCover}: 
\begin{figure}[h]
\centering
\includegraphics[scale = .5]{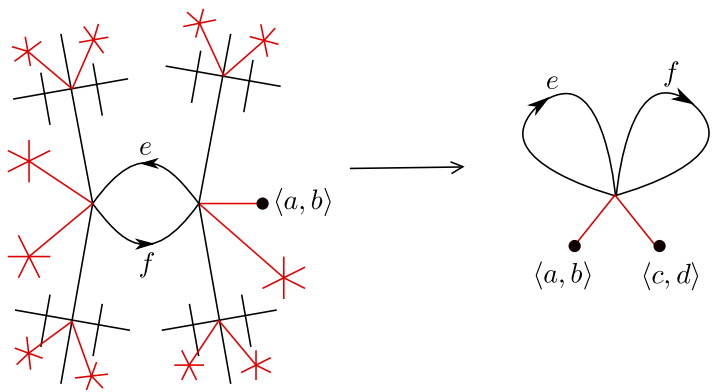}
\caption{Infinite sheeted cover}
\label{F:InfniteCover}
\end{figure}
\item A finite sheeted cover whose fundamental group contains $H=\la a,b,ef \ra$ is shown in Figure~\ref{F:FiniteCover}: 
\begin{figure}[h]
\centering
\includegraphics[scale = .5]{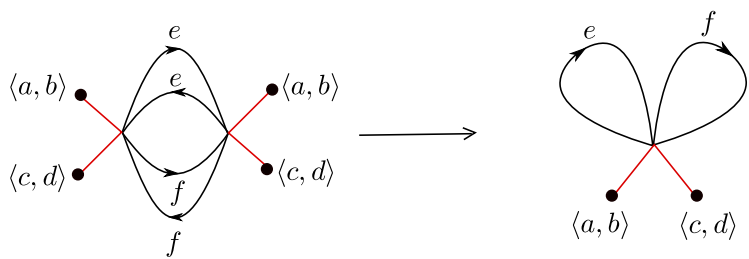}
\caption{Finite sheeted cover}
\label{F:FiniteCover}
\end{figure}
\end{itemize}

\begin{lemma}\label{L:WhGraph} Let $\phi_0: G \to G$ be an $\ffa$-train track representative of an automophism that is fully irreducible relative to $\ffa$. 
\begin{enumerate}[(a)]
\item The relative Whitehead graph of $\phi_0$ is connected at each vertex of $G$. 
\item Suppose $p: G^{\p\p} \to G$ is a finite sheeted covering space such that for every vertex $v$ of $G^{\p\p}$, $p_{\ast}(\op{Stab}(v)) = \op{Stab}(p(v))$, and $\phi^{\p\p}: G^{\p\p} \to G^{\p\p}$ is a lift of $\phi_0: G \to G$. Then the relative Whitehead graph of $\phi^{\p\p}$ at a vertex $v$ of $G^{\p \p}$ is the lift of the relative Whitehead graph of $\phi_0$ at $p(v)$ and in particular is connected. 
\item Let $H$ be a finitely generated subgroup of $\free$ such that for every $[A] \in \ffa$ either $H \cap A$ is trivial or equal to $A$, up to conjugation. If $H$ carries $\lamination$ then $H$ has finite index in $\free$. 
\end{enumerate}\end{lemma}
\begin{proof}
\begin{enumerate}[(a)]
\item The same proof as in the absolute case works by doing a blow-up construction (\cite[Proposition 4.5]{BH:TrainTracks}) at a vertex. 

\item The graph $G^{\p\p}$ gets a legal turn structure from the lift of $G$ and it gets a legal turn structure from the map $\phi^{\p \p}$. We have to show that a turn in $G^{\p\p}$ whose image in $G$ is $\lamination$-legal is in fact crossed by a lift of a leaf of $\lamination$ to $G^{\p\p}$.  

\begin{enumerate}[(i)]
\item Let $a^{\p \p}, b^{\p \p}$ be two edges incident at a vertex $v^{\p \p}$ of $G^{\p \p}$ where $p(a^{\p \p}) = a$ and $p(b^{\p \p}) = b$ are such that $ab$ is a legal turn at $p(v^{\p \p}) = v$ in $G$. The same proof as \cite[Lemma 2.1]{BFH:Laminations} works in this case. 

\item Let $v$ be a vertex of $G^{\p \p}$ with non-trivial vertex stabilizer. Let $a^{\p \p}$ be an edge at $v^{\p\p}$ such that $a = p(a^{\p \p})$. Let $a'$ be the pre-image of $a$ in $G'$ such that there exists an edge $e' \in G'$ such that $\phi_0'(e') = \ldots a' w \ldots$ where $w$ is a path in $G_{r-1} \subset G'$. After passing to a power, we have $\phi_0'(a') = \ldots a' w \ldots$. Thus $a'$ and hence $a$ has a fixed point. Now by the same argument as in the previous case we get that $\phi^{\p \p}(a^{\p \p})$ maps over $a^{\p \p} w^{\p \p}$.

\end{enumerate} 

\item Let $\Gamma_H$ be the core of the covering space of $G$ corresponding to a subgroup $H$ as in the statement of the lemma. Here $\Gamma_H$ is a finite graph. Let $i: \Gamma_H \to G$ be the isometric immersion. If $H$ has infinite index in $\free$ then we can add more vertices and edges to $\Gamma_H$ to complete it to a finite sheeted covering $\Gamma_H^{\prime}$ of $G$. We can pass to a further finite sheeted cover $\Gamma_H^{\prime \prime}$ such that $\phi_0: G \to G$ lifts to a map $\phi^{\p\p}: \Gamma_H^{\prime \prime} \to \Gamma_H^{\prime \prime}$. 
By the previous part we have that the relative Whitehead graph is connected at every vertex of $\Gamma_H^{\prime \prime}$. Therefore lifts of the leaves of $\lamination(G)$ cross every edge of $\Gamma_H^{\prime \prime}$. Under the projection $p: \Gamma_H^{\prime \prime} \to \Gamma_H^{\prime}$ we see that the edges we added to $\Gamma_H$ are crossed by leaves of $\lamination$ so $H$ does not carry $\lamination$. 
\qedhere \end{enumerate}
\end{proof}
\subsection{Transverse covering}\label{subsec:Transverse covering}

Let $\phi_0: G \to G$ be an $\ffa$-train track representative of a relative fully irreducible automorphism $\oo$. Let $\phi: \RelTTTree \to \RelTTTree$ be a lift to the universal cover $T_G$ of $G$. In this section we define a transverse covering for $\RelTTTree$ which will be used in the proof of Lemma~\ref{L:ConnectingEdges}.  

We define an equivalence relation on $\lamination(\RelTTTree)$ as follows: two leaves $\gamma, \gamma^{\prime}$ are equivalent if there is a sequence of leaves $\gamma=\gamma_1, \gamma_2, \ldots, \gamma_n=\gamma^{\prime}$ such that $\gamma_i$ and $\gamma_{i+1}$ overlap in a non-trivial edge path in $\RelTTTree$. Let $\mathcal{Y}(\lamination) = \{Y_i\}_{i \in I}$ be the set of subtrees of $T_G$ such that $Y_i$ is the realization of leaves of $\lamination(\RelTTTree)$ in an equivalence class. 


\begin{definition}[Closed subtree {\cite[Definition 2.4]{G:LimitGroups}}]  A subtree $Y$ of a tree $T$ is called closed if the intersection of $Y$ with any segment of $T$ is either empty or a segment of $T$. \end{definition}

\begin{definition}[Transverse Covering {\cite[Definition 4.6]{G:LimitGroups}}]  A transverse covering of an $\mathbb{R}$-tree $T$ is a family $\mathcal{Y}$ of non-degenerate closed subtrees of $T$ such that every arc in $T$ is covered by finitely many subtrees in $\mathcal{Y}$ and any two distinct subtrees in $\mathcal{Y}$ intersect in at most one point. \end{definition}

\begin{lemma} The set $\mathcal{Y}(\lamination)$ forms a transverse covering of $\RelTTTree$. \end{lemma}
\begin{proof} Since an element $Y$ of $\mathcal{Y}(\lamination)$ contains the realization in $\RelTTTree$ of a leaf of $\lamination$, $\mathcal{Y}(\lamination)$ is a covering of $\RelTTTree$. We now need to check that every arc of $\RelTTTree$ is covered by finitely many $Y_i$. Indeed, if an edge of $\RelTTTree$ is covered by multiple $Y_i$ then by the definition of the equivalence relation they are connected. Therefore an edge of $\RelTTTree$ is covered by one subtree $Y_i$ and a finite arc is covered by finitely many subtrees in $\mathcal{Y}(\lamination)$. Also by definition two distinct subtrees $Y_i, Y_j$ intersect in at most one point. \end{proof}


\begin{example}
Recall the automorphism $\oo$ from Example~\ref{E:Disconnected Relative Wh Graph} given by 
$\oo(a) = ab, 
\oo(b) = b,  
\oo(c) = cad, 
\oo(d) = dcad$.
We say two leaves in $\lamination(\AbsTree)$ are equivalent if they overlap in an edge in the top EG stratum. There are two different equivalence classes of leaves at a vertex in the universal cover $\AbsTree$. 
See Figure~\ref{F:ExEqClasses}.
\begin{figure}[h]
\centering
\includegraphics[scale=.7]{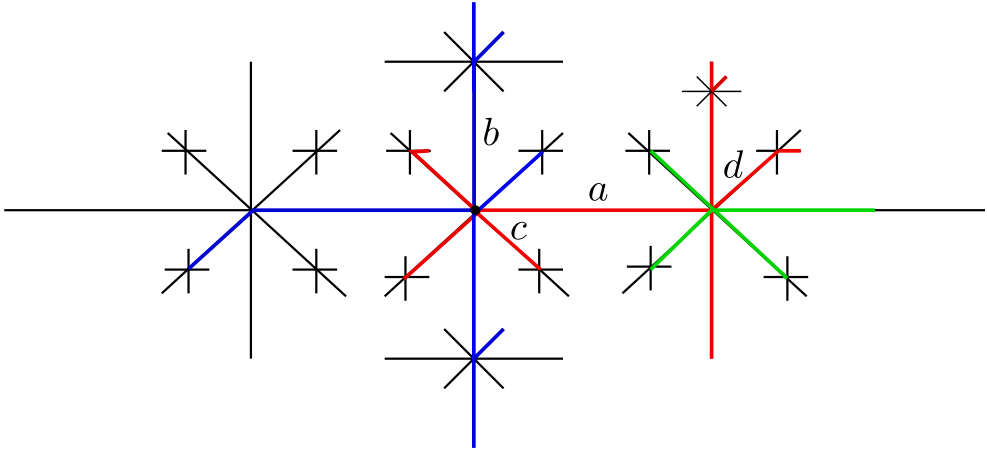}
\caption{Three different equivalence classes in $\AbsTree$}
\label{F:ExEqClasses}
\end{figure}

By collapsing the edges with labels $a$ and $b$ in $G^{\p}$ we get a relative rose $G$ with two petals and a non-trivial vertex stabilizer. The covering of $\AbsTree$ in Figure~\ref{F:ExEqClasses} descends to a transverse covering of $\RelTTTree$. See Figure~\ref{F:ExEqClasses1}. \begin{figure}[h]
\centering
\includegraphics[scale=.7]{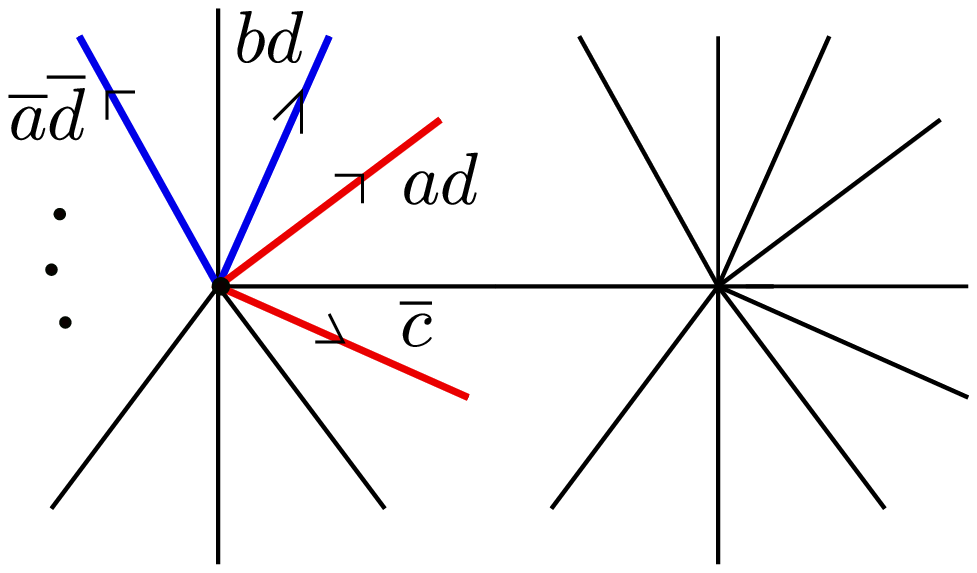}
\caption{Different equivalence classes in $\RelTTTree$}
\label{F:ExEqClasses1}
\end{figure}

\end{example}

\subsection{$\mathcal{Q}$ map}\label{subsec:Q map}
Given a tree $T$ with dense orbits in $\overline{CV}_n$,  in \cite{LL:NSDynamics}, Levitt and Lustig define a map called the $\mathcal{Q}$ map from the boundary of $\free$ to $\overline{T} \cup \partial T$, where $\overline{T}$ is the metric completion of $T$. This map is the key tool used to prove north-south dynamics for a fully irreducible automorphism on the closure of outer space. We will follow the same techniques to get a relative result. The main proposition in this section is Proposition~\ref{P:5.1}. 

Let $T_0$ be a metric simplicial $\free$-tree. Let $v(T_0)$ denote the volume of the quotient graph $T_0/\free$. Let $T$ be a metric minimal very small $\free$-tree and let $\overline{T}$ be the metric completion of $T$. 
Let $T$ be an $(\free,\ffa)$-tree. The boundary of $T$, denoted $\partial T$, is defined as the set of infinite rays $\rho : [0,\infty) \to T$ up to an equivalence. Namely, two rays are equivalent if they intersect along a ray. If $T_0$ is a Grushko $(\free, \ffa)$-tree then there is a canonical identification between $\partial \free \setminus \partial \ffa$ (see Section~\ref{subsec:Boundary} for definition) and $\partial T_0$. We denote by $\rho$ a ray in $T_0$ representing the point $X$ in $\partial T_0$. Given an equivariant map $h : T_0 \to T$, let $r=h(\rho)$. We say \emph{$X$ is $T$-bounded} if $r$ is bounded in $T$ (this does not depend on the choice of $h$ as shown in \cite[Proposition 3.1]{LL:NSDynamics}). If $r$ is unbounded then we get a ray representing a point in $\partial T$. 

Let $h: T_0 \to T$ be a continuous map between $\mathbb{R}$-trees. We say $h$ has \emph{bounded cancellation property} if there exists a constant $C\geq 0$ such that the $h$-image of any segment $pq$ in $T_0$ is contained in the $C$ neighborhood of the geodesic joining $h(p)$ and $h(q)$ in $T$. The smallest such $C$ is called the bounded cancellation constant for $h$, denoted $\op{BCC}(h)$. The following fact about BCC for very small trees is a generalization of Cooper's bounded cancellation lemma \cite{C:BCC}, and can be found in \cite[Lemma 3.1]{BFH:Laminations} and \cite{GJLL:Index}. 
\begin{lemma}
Let $T$ be an $\mathbb{R}$-tree with a minimal very small action of $\free$. Let $T_0$ be a free simplicial $\free$-tree, and $h:T_0 \to T$ an equivariant map. Then $h$ has bounded cancellation, with $\op{BCC}(h) \leq \op{Lip}(h) v(T_0)$, where $\op{Lip}(h)$ is the Lipschitz constant for $h$.  
\end{lemma}
\begin{proposition}[Small BCC]\label{P:SmallBBT}
Let $T \in \relCVClo$ be a minimal $\free$-tree with dense orbits and trivial arc stabilizers. Given $\epsilon >0$, there exists an $(\free, \ffa)$-tree $T_0 \in \relCV$, $v(T_0) < \epsilon$, and an equivariant map $h : T_0 \to T$ whose restriction to each edge is isometric and $BCC(h)< \epsilon$.  \end{proposition}

The proof of the above proposition when $T \in \overline{CV}_n$ and $T_0 \in CV_n$ in \cite[Proposition 2.2]{LL:NSDynamics} starts with an equivariant map $h: T_0 \to T$ which is isometric on edges. Then given an edge $e$ of $T_0$, one replaces $h$ by $h^{\prime}: T_0^{\prime} \to T$ with $v(T_0^{\prime}) \leq v(T_0) - 1/6|e|$. If $T \in \relCVClo$ then we can start with an equivariant map $h: T_0 \to T$ isometric on edges where $T_0 \in \relCV$ and do the same argument.    

\begin{proposition}[$\mathcal{Q}$ map]\label{P:Q map}
Let $T \in \relCVClo$ be a minimal $(\free, \ffa)$-tree with dense orbits and trivial arc stabilizers. Suppose $X \in \partial \free \setminus \partial \ffa$ is $T$-bounded. Then there is a unique point $\mathcal{Q}(X) \in \overline{T}$ such that for any equivariant map $h : T_0 \to T$ and any ray $\rho$ representing $X$ in $T_0 \in \relCV$, the point $\mathcal{Q}(X)$ belongs to the closure of $h(\rho)$ in $\overline{T}$. Also, every $h(\rho)$ is contained in a 2BCC($h$)-ball centered at $\mathcal{Q}(X)$, except for an initial part. \end{proposition}

In \cite[Proposition 3.1]{LL:NSDynamics} the above lemma is proved for any tree with dense orbits in the closure of outer space hence it applies to our setting as well. Since the free factors in $\ffa$ are elliptic in $T$ we can take the tree $T_0$ in the original proof to be such that $T_0 \in \relCV$. Also by \cite[Remark 3.7]{LL:NSDynamics}, if $\mathcal{Q}(X) = \mathcal{Q}(X^{\p})$ for a bi-infinite geodesic $\gamma$ with end points $X, X^{\p}$ then $h(\gamma)$ lies in a $2BCC(h)$-neighborhood of $\mathcal{Q}(X)$.  

\begin{definition}[Dual lamination of a tree  \cite{CHL:DualLaminationsII}]\label{D:QLamination} Let $T$ be a tree with dense orbits in $\partial CV_n$: $$L_{\mathcal{Q}}(T) = \{\{X, X^{\p}\} \in \partial^2 \free | \, \mathcal{Q}(X) = \mathcal{Q}(X^{\p})\}.$$ It is shown in \cite{CHL:DualLaminationsII} that $L_{\mathcal{Q}}(T)$ is the same as $L(T)$ (see Section~\ref{subsec:Laminations} for definition). 

For an algebraic lamination $L$, let \emph{support} $s(L) \subset \partial \free \setminus \partial \ffa$ be the set of all $X \in \partial \free$ such that $L$ contains some pair $\{ X, X^{\prime}\}$. The laminations $L_{\mathcal{Q}}(\StableTree)$ and $L_{\mathcal{Q}}(\UnstableTree)$ are $\free$-invariant and $\oo$-invariant. 

\end{definition}

\begin{definition}[Eigenray]
Let $f_0: \tau \to \tau$ be a relative train track map or an $\ffa$-train track map. Let $f: T_{\tau} \to T_{\tau}$ be a lift of $f_0$ to the universal cover $T_{\tau}$ of $\tau$. Let $v_0$ be a fixed vertex in $\tau$ with a fixed direction $e$, where $e$ is an edge in an EG stratum. Let $v$ be a lift of $v_0$ to $T_{\tau}$. Then a lift based at $v$ of the ray $\lim_{n \to \infty}f_0^n(e)$ is called an \emph{eigenray of $f$ based at $v$}, denoted by $X_v \in \partial T_{\tau}$.  
\end{definition}

Recall from Definition~\ref{D: A TrainTrack} the $\ffa$-train track map $\phi_0: G \to G$ representing $\oo$ and a lift to the universal cover $\phi: \RelTTTree \to \RelTTTree$. 
Let $E\lamination$ be the set of all eigenrays of $\phi$.

\begin{remark} In the absolute case of a fully irreducible automorphism, any eigenray is in fact a half-leaf of $\lamination$, that is, it is contained in a generic leaf of $\lamination$. Thus it suffices to consider points in $s(\lamination)$ for the proof of \cite[Lemma 5.2]{LL:NSDynamics}. In the relative case, an eigenray based at a vertex with trivial stabilizer is a half-leaf of $\lamination$ but an eigenray based at a vertex with non-trivial vertex stabilizer might not be a half-leaf of $\lamination$. 
It will be a half-leaf of a diagonal leaf of $L_{\mathcal{Q}}(\UnstableTree)$ as explained below.  \end{remark} 

\begin{lemma}\label{L:eigenrays}
$s(L_{\mathcal{Q}}(\UnstableTree))$ contains $s(\lamination)$ and $E\lamination$. \end{lemma}
\begin{proof}
The statement that $s(L_{\mathcal{Q}}(\UnstableTree))$ contains $s(\lamination)$ follows from Lemma~\ref{L:DualLamination} where it is shown that $L_{\mathcal{Q}}(\UnstableTree)$ contains $\lamination$. Let $R_v: \mathbb{R}^+ \to T_G$ be a ray representing an eigenray $X_v$ of $\phi$ based at a vertex $v$ of $T_G$ with non-trivial stabilizer. Let $R_v(\infty) = X_v \in \partial T_G$, which is identified with a point in $\partial \free$, also denoted by $X_v$. Let $\nu: T_G \to \UnstableTree$ be an $\free$-equivariant map. 

We first show that $\nu(R_v)$ is $\UnstableTree$-bounded. Suppose not. Then for every $C>0$ and every $t_0>0$ there exist $t_2>t_1>t_0$ such that $d_{\UnstableTree}(\nu(R_v(t_2)), \nu(R_v(t_1))) > C$. Now choose $C > 2BCC(\nu)$. Since $R_v$ is an eigenray, a generic leaf $l^+$ of $\lamination$ crosses the segment $\sigma_v = [R_v(t_2), R_v(t_1)]$ of $R_v$. By \cite[Remark 3.7]{LL:NSDynamics}, the $\nu$ image of $l^+ = \{X, X^{\p}\}$ is in a $2BCC(\nu)$ neighborhood of $\mathcal{Q}(X) = \mathcal{Q}(X^{\p})$. This implies that the diameter of $\sigma_v$ under $\nu$ is less than $2BCC(\nu)$, which is a contradiction. 

Next we want to prove that $\mathcal{Q}(X_v) = \tilde{v}$ where $\tilde{v}$ is the point in $\UnstableTree$ whose stabilizer contains the stabilizer of $v$. Given $\epsilon>0$, let $h: T_0 \to \UnstableTree$ be an $\free$-equivariant map with $BCC(h) < \epsilon$ as given by Proposition~\ref{P:SmallBBT}. Let $\mu: T_G \to T_0$ be an $\free$-equivariant map and let $\nu= h \circ \mu$. Let $\overline{R}_v = \mu(R_v)$. Then by Proposition~\ref{P:Q map}, $h(\overline{R}_v)$ is contained in a $2BCC(h)$ neighborhood of $\mathcal{Q}(X_v)$ except an initial segment. Suppose $\mathcal{Q}(X_v) \neq \tilde{v}$. 
There exists a $g \in \free \setminus \ffa$ for which the following is true: 
let $\sigma_g$ be the subsegment of $R_v$ joining $v$ and $gv$ such that the length of $\overline{\sigma}_g := \mu(\sigma_g)$ is non-zero and $h(\overline{\sigma}_g)$ is not contained in a $2BCC(h)$-neighborhood of $\mathcal{Q}(X_v)$. Since $R_v$ is an eigenray it contains translates of the segment $\sigma_g$. There exists some translate $\sigma_g^{\p}$ of $\sigma_g$ joining points $u, gu$ on $R_v$ such that $h(\overline{\sigma}_g^{\p})$, where $\overline{\sigma}_g^{\p}:=\mu(\sigma_g^{\p})$, is in a $2BCC(h)$-neighborhood of $\mathcal{Q}(X_v)$ because $h(\overline{R}_v)$ is $\UnstableTree$-bounded. 
But $g$ acts by isometries on $\UnstableTree$ so the diameters of $h(\overline{\sigma}_g)$ and $h(\overline{\sigma}_g^{\p})$ cannot be different. Thus $\tilde{v}$ is in a $2BCC(h)$-neighborhood of $\mathcal{Q}(X_v)$. Since $\epsilon$, which bounds $BCC(h)$, was arbitrary we have that $\mathcal{Q}(X_v) = \tilde{v}$.     


Now we show that for every vertex $v$ of $T_G$ with non-trivial stabilizer there are at least two eigenrays $X_v, X_v^{\p}$ based at $v$. This will imply that $\{X_v, X_v^{\p}\} \in L_{\mathcal{Q}}(\UnstableTree)$ and hence $E\lamination \subset s(L_{\mathcal{Q}}(\UnstableTree))$. If the image of $v$ in $G = \RelTTTree/ \free$ has at least two gates then each gate will have a fixed direction which gives us different eigenrays based at $v$. If there is only one gate at $v$ then in $T_G$ the orbit of a given ray $R_v$ under the stabilizer of $v$ gives distinct eigenrays based at $v$. 
 \end{proof}

\begin{remark}
From the above proposition we get that the following two types of leaves are contained in $L_{\mathcal{Q}}(\UnstableTree)$:

\begin{enumerate}[(a)]
\item leaves of the lamination $\lamination$, which we call $\lamination$-leaves, and,  
\item leaves obtained by concatenating two eigenrays, which are called \emph{diagonal} leaves.
\end{enumerate}
\end{remark}

The next proposition, which is the relativization of \cite[Proposition 5.1]{LL:NSDynamics}, is the main technical proposition of this section. 

\begin{proposition} \label{P:5.1}
If $T\in \relCVClo$ is a minimal $(\free, \ffa)$-tree with dense orbits and trivial arc stabilizers then at least one of the following is true: 
\begin{enumerate}[(a)]
\item there exists a generic leaf $\{X, X^{\prime} \}$ of $\lamination$ or $\Rlamination$ such that $\mathcal{Q}(X) \neq \mathcal{Q}(X^{\prime})$, 
\item  there exists a diagonal leaf $\{X, X^{\prime} \}$ of $L_{\mathcal{Q}}(\UnstableTree)$ or $L_{\mathcal{Q}}(\StableTree)$ such that $\mathcal{Q}(X) \neq \mathcal{Q}(X^{\prime})$.
\end{enumerate} 
\end{proposition}

Since diagonal leaves are obtained by concatenating eigenrays, $(b)$ implies $(a)$ in the above proposition. Morally, the above proposition says that if $T\in \relCVClo$ is a minimal $(\free, \ffa)$-tree with dense orbits such that $L_{\mathcal{Q}}(T)$ contains both $L_{\mathcal{Q}}(\StableTree)$ and $L_{\mathcal{Q}}(\UnstableTree)$ then $T$ is in fact a trivial tree. 
The proof of the proposition depends on Lemma~\ref{L:5.2} and Lemma~\ref{L:5.3}. We need the following lemma for the proof of Lemma~\ref{L:5.2}. 

\begin{lemma}\label{L:ConnectingEdges}
If $e, e^{\p}$ are edges with a common initial vertex $v$ in $\RelTTTree$, then there exists a sequence $e_0 = e, e_1, \ldots, e_k = e^{\p}$ of distinct edges starting at $v$ such that every edge path $\overline{e_i}e_{i+1}$ is crossed by either a $\lamination$-leaf or a diagonal leaf of $L_{\mathcal{Q}}(\UnstableTree)$.   
\end{lemma}
\begin{proof}
If the vertex stabilizer of $v$ is trivial then by Lemma~\ref{L:WhGraph} the Whitehead graph of $\lamination$ is connected at the vertex $v$. Hence the lemma follows by using the $\lamination$-leaves of $L_{\mathcal{Q}}(\UnstableTree)$. Now let's assume that the vertex stabilizer of $v$ is non-trivial. Consider the transverse covering $\mathcal{Y}(\lamination)$ of $\RelTTTree$ from Section~\ref{subsec:Transverse covering}. Since an element $Y$ of $\mathcal{Y}(\lamination)$ contains a generic leaf of $\lamination$, $Y$ crosses the $\free$-orbit of every edge in $\RelTTTree$. Let $Y_e$ and $Y_{e^{\p}}$ be the elements of $\mathcal{Y}(\lamination)$ that contain $e$ and $e^{\p}$ respectively. Let $E, E^{\p}$ be the set of edges with initial vertex $v$ which are in $Y_e$ and $Y_{e^{\p}}$ respectively. 

If $Y_e$ is equal to $Y_{e^{\p}}$ then the lemma follows by using $\lamination$-leaves in $L_{\mathcal{Q}}(\UnstableTree)$. 
Suppose $Y_e \neq Y_{e^{\p}}$. Let $p: \RelTTTree \to G$ be the quotient map by the action of $\free$. Every gate at the vertex $\pi(v)$ has a fixed direction. Thus we can find an eigenray $X$ in $\RelTTTree$ based at $v$ with initial edge $f$ in $E$ (since $Y_e$ crosses $\free$-orbit of every edge at $v$). Similarly, we get an eigenray $X^{\p}$ based at $v$ and initial edge $f^{\p}$ in $E^{\p}$. The diagonal leaf $\{X, X^{\p}\}$ of $L(\UnstableTree)$ crosses $\overline{f} f^{\p}$. Now we have a sequence of edges $e_0 = e, e_1, \ldots, e_l=f, e_{l+1}=f^{\p}, e_{l+2}, \ldots, e_k = e^{\p}$ starting at $v$ such that every edge path $\overline{e_i}{e_{i+1}}$ for $i \neq l$ is crossed by a $\lamination$-leaf and $\overline{e_l}e_{l+1}$ is crossed by a diagonal leaf.  
\end{proof}

\begin{lemma}\label{L:5.2}
Suppose $\mathcal{Q}(X) = \mathcal{Q}(X^{\prime})$ for every generic leaf $\{X, X^{\p}\}$ of $\lamination$ and for every diagonal leaf $\{X, X^{\p}\}$ of $L_{\mathcal{Q}}(\UnstableTree)$. Let $Z, Z^{\prime}$ belong to $s(\lamination) \cup E\lamination$. Then the distance in $\overline{T}$ between $\mathcal{Q}(\oo^p(Z))$ and $\mathcal{Q}(\oo^p(Z^{\prime}))$ tends to 0 as $ p \to + \infty$.\end{lemma} 

\begin{proof}
We follow the proof of Lemma 5.2 in \cite{LL:NSDynamics}. If $Z$ is in $s(\lamination)$ then there exists a ray $\rho$ in $\RelTTTree$ contained in $\lamination(\RelTTTree)$ with end point $Z$. If $Z$ is in $E\lamination$ then there exists an eigenray $\rho$ of $\phi$ with end point $Z$. Let's suppose $Z \in E\lamination$ and $Z^{\p} \in s(\lamination)$ with corresponding rays $\rho$ and $\rho^{\p}$ to exhibit the proof in both cases. Let $e, e^{\prime}$ be the initial edges of the two rays $\rho$ and $\rho^{\p}$. By Lemma~\ref{L:ConnectingEdges} we can find a sequence of edges $e=e_0, e_1, e_2, \ldots, e_k = e^{\prime}$, in $\RelTTTree$ connecting $e$ to $e^{\prime}$ such that the finite subpaths $\gamma_i = e_i e_i^{\prime}$ are subpaths of either $\lamination$-leaves or diagonal leaves of $L_{\mathcal{Q}}(\UnstableTree)$ where $e_i^{\prime}$ is the same as $e_{i+1}$ but not necessarily with the same orientation. Note that the union of $\gamma_i$ and $\gamma_{i+1}$ is either a tripod or a segment of length 3. The rest of the proof follows exactly as in \cite[Lemma 5.2]{LL:NSDynamics} 
\end{proof}

The following lemma is the relativization of \cite[Proposition 5.3]{LL:NSDynamics}. 
Recall the $\ffa$-train track map $\phi_0: G \to G$, and a lift to the universal cover $\phi: \RelTTTree \to \RelTTTree$ representing $\oo$ where $T_G \in \relCV$. 

\begin{lemma}\label{L:5.3}
Suppose $\mathcal{Q}(X) = \mathcal{Q}(X^{\prime})$ for every generic leaf $\{X, X^{\p}\}$ of $\lamination$ and for every diagonal leaf $\{X, X^{\p}\}$ of $L_{\mathcal{Q}}(\UnstableTree)$.
Then there exist maps $i_p : \RelTTTree \to \overline{T}$, $p \in \mathbb{N}$ such that $i_p \circ \phi^p$ is $\free$-equivariant and $BCC(i_p) \to 0$ as $p \to \infty$.\end{lemma}

\begin{proof}
We can assume that there are no vertices  with trivial stabilizer in $\RelTTTree$. If there were some such vertices we could collapse a tree in $\RelTTTree / \free$ and factor through the quotient of $\RelTTTree$. For a representative $v$ of an orbit of vertices in $\RelTTTree$ fix an eigenray $X_v$ in $E\Rlamination$ such that $\mathcal{Q}(X_v) = \tilde{v}$, where $\tilde{v}$ is a point in $T$ whose stabilizer contains the stabilizer of $v$. Then $\free$-equivariantly assign an eigenray to every vertex in the orbit of $v$. In this way, assign an eigenray to each vertex of $\RelTTTree$. 

We will now define a map $i_p: \RelTTTree \to \overline{T}$ and show that $i_p(e) \to 0$ as $p \to \infty$ for every edge $e$ of $\RelTTTree$. For a vertex $v \in \RelTTTree$, set $i_p(v) = \mathcal{Q}(\oo^{-p}(X_v))$ and extend linearly on edges. Now for an edge $e$ of $\RelTTTree$ with end points $v, u$ we have, by applying Lemma~\ref{L:5.2} to $\oo^{-1}$, that distance between $i_p(v) = \mathcal{Q}(\oo^{-p}(X_v))$ and $i_p(u)=\mathcal{Q}(\oo^{-p}(X_u))$ goes to zero as  $p \to \infty$. Thus $i_p(e) \to 0$ which implies that $\op{BCC}(i_p) \to 0$. The map $i_p$ satisfies a twisted equivariance relation $g \circ i_p = i_p \circ \oo^p(g)$ for all $g \in \free$. 

Also $i_p \circ \phi^p$ is $\free$-equivariant. Indeed, $ g \circ i_p \circ \phi^{p} = i_p \circ \oo^{p}(g) \circ \phi^{p} =  i_p \circ \phi^p \circ g.$

\end{proof}
\begin{proof}[Proof of Proposition \ref{P:5.1}] Assume by contradiction that $\mathcal{Q}(X) = \mathcal{Q}(X^{\prime})$ for every generic leaf $\{ X, X^{\prime}\}$ of $\lamination$ and $\Rlamination$ and every diagonal leaf of $L_{\mathcal{Q}}(\UnstableTree)$ and $L_{\mathcal{Q}}(\StableTree)$. Let $e$ be an edge in $\RelTTTree$ and let $\gamma \in \lamination$ be a leaf that crosses $e$. Then $\phi^p(\gamma)$ is also a leaf of the lamination. By assumption, the end points of $\gamma$ map to the same point under the $\mathcal{Q}$ map. By Proposition \ref{P:Q map} and \cite[Remark 3.7]{LL:NSDynamics}, $(i_p \circ \phi^p) (\gamma)$ is contained in a ball of radius $2BCC(i_p \circ \phi^p)$ in $\overline{T}$. We have $BCC(i_p \circ \phi^p) \leq BCC(i_p)+ Lip(\phi^p)BCC(\phi^p)$. Since $\gamma$ is a leaf of $\lamination$, $\phi^p$ restricted to $\gamma$ has no cancellation thus we get that $(i_p \circ \phi^p) (\gamma)$ is in fact contained in a ball of radius $2BCC(i_p)$ in $\overline{T}$.  
Thus the diameter of $(i_p\circ \phi^p)(e)$ in $\overline{T}$ is bounded by $4BCC(i_p)$. 

Now let $u$ be a conjugacy class, represented by a loop of edge-length $k$ in $G=\RelTTTree / \free$. Since $i_p \circ \phi^p$ is $\free$-equivariant, the translation length of $u$ in $T$ is bounded by $4kBCC(i_p)$. Since $BCC(i_p) \to 0$ as $p \to \infty$, we get that every $u$ has zero translation length in $T$ which is a contradiction. 
\end{proof}

\subsection{Proof of Theorem C}\label{subsec:NS for trees proof}
We will now put together the results from  Section~\ref{subsec:Relative Wh graph} and Section~\ref{subsec:Q map} to prove the following lemma, which shows that the conditions mentioned in Proposition~\ref{P:Proof Strategy} are satisfied by all trees in $\relCVClo$ if we allow $\gamma$ to be a leaf of $\lamination$ or $\Rlamination$. 

\begin{lemma}\label{L:Tcases}
Let $T \in \relCVClo$. Then there exists a tree $T_0 \in \relCV$, an equivariant map $h : T_0 \to T$, and a bi-infinite geodesic $\gamma_0 \subset T_0$ representing a generic leaf $\gamma$ of $\lamination$ or $\Rlamination$ such that $h(\gamma_0)$ has diameter greater than 2BCC($h$). \end{lemma}

\begin{proof}
 There are three cases to consider for a tree $T$ in $\relCVClo$. 
\begin{itemize}
\item \emph{$T$ has dense orbits} (which implies that arc stabilizers are trivial by \cite[Lemma 4.2]{LL:NSDynamics}): Proposition~\ref{P:5.1} provides either a generic leaf $\{X, X^{\prime} \}$ in $\lamination$ or $\Rlamination$ with $\mathcal{Q}(X) \neq \mathcal{Q}(X^{\prime})$, or it provides an eigenray $X_v \in E\lamination$ or $E\Rlamination$ based at a vertex $v$ of $\RelTTTree$ such that $\mathcal{Q}(X_v) \neq \tilde{v}$, where $\tilde{v}$ is the vertex of $T$ containing the stabilizer of $v$. We can choose $h: T_0 \to T$ with $2BCC(h)< d(\mathcal{Q}(X), \mathcal{Q}(X^{\prime}))$ or $2BCC(h) < d(\mathcal{Q}(X_v), \tilde{v})$ using Proposition~\ref{P:SmallBBT}. In the first case, we let $\gamma_0$ be the geodesic joining end points corresponding to $X, X^{\prime}$ in $T_0$.  In the second case, there exists a subsegment of an eigenray $R_v$ corresponding to $X_v$ whose diameter in $T$ is at least $d_T(\mathcal{Q}(X_v), \tilde{v})$. We choose $\gamma_0$ to be any generic leaf (of either $\lamination$ or $\Rlamination$) crossing that subsegment. 

\item \emph{$T$ does not have dense orbits and is also not simplicial:} then $T$ contains simplicial parts and also subtrees $T_v$ with the property that some subgroup $G_v \subset \free$ acts with dense orbits on $T_v$. Let $\pi : T \to T^{\prime}$ be a collapse map such that $T^{\prime}$ has dense orbits. Choose $\gamma_0$ as in the previous case, using $T^{\prime}$. Then by Proposition~\ref{P:Proof Strategy} $\gamma_0$ is unbounded in $T^{\p}$ and hence it is $T$-unbounded. The map $h : T_0 \to T$ may be chosen arbitrarily.  
\item \emph{$T$ is simplicial:} we want to show that a generic leaf of $\lamination$ is unbounded in $T$. 
We need to show that a tail of a generic leaf of $\lamination$ or $\Rlamination$ does not live in $\partial B$ for any vertex stabilizer $B$. By \cite[Corollary III.4]{GL:RankOfActions} vertex stabilizer in a tree in $\overline{CV}_n$ is finitely generated and has infinite index in $\free$.  Also given $T$ in $\relCVClo$, for every $[A] \in \ffa$ a vertex stabilizer in $T$ either contains the full free factor $A$ or intersects it trivially. 
By Lemma~\ref{L:WhGraph}, a generic leaf of the attracting lamination cannot be carried by a vertex stabilizer of $T$, therefore it is unbounded in $T$. One can choose $h : T_0 \to T$ arbitrarily.  \qedhere
\end{itemize}      
\end{proof}

\begin{thmC}  
 Let $\ffa$ be a non-trivial free factor system of $\free$ such that $\zeta(\ffa) \geq 3$. Let $\oo \in \relOut$ be fully irreducible relative to $\mathcal{A}$. Then $\oo$ acts on $\relCVClo$ with uniform north-south dynamics: there are two fixed points $\StableTree$ and $\UnstableTree$ and for every compact set $K$ of $\relCVClo$ that does not contain $\UnstableTree$ (resp. $\StableTree$) and for every open neighborhood $U$ (resp. $V$) of $\StableTree$ (resp. $\UnstableTree$), there exists an $N\geq 1$ such that for all $n \geq N$ we have $K.\oo^n \subseteq U$ (resp. $K.\oo^{-n} \subseteq V$) 
\end{thmC}
\begin{proof}
By Lemma~\ref{L:Tcases} and Proposition~\ref{P:Proof Strategy}, every $T$ in $\relCVClo$ converges either to $\StableTree$ under forward iterates or to $\UnstableTree$ under backward iterates. We know that $\StableTree$ is locally attracting and $\UnstableTree$ is locally repelling. Thus given a tree $T \neq \UnstableTree$, the set of its limit points under forward iterates cannot contain the repelling point $\UnstableTree$ and hence $T$ converges to $\StableTree$. Similarly, a tree $T \neq \StableTree$ under backward iterates converges to $\UnstableTree$. By Proposition~\ref{P:Proof Strategy}, for every $T \neq \UnstableTree$ there exists a neighborhood $V_T$ of $T$ such that $V_T$ uniformly converges to $\StableTree$ under $\oo$. Let $\mathcal{V}$ be an open cover of $\relCVClo \setminus \UnstableTree$ by open sets of the form $V_T$. Let $K$ be a compact set in $\relCVClo \setminus \UnstableTree$. Then $K \cap \mathcal{V}$ is an open cover of $K$ which has a finite subcover by sets of the form $V_T$. Thus $K$ uniformly converges to $\StableTree$ under $\oo$. Alternatively, since $\relCVClo$ is compact, by \cite{HK:UniformNS}, pointwise north-south dynamics implies uniform north-south dynamics.  
\end{proof}

\section{Intersection Form}\label{sec:IntersectionForm}
In \cite{KL:IntersectionForm}, Kapovich and Lustig established an intersection form between $\overline{cv}_n$, the closure of unprojectivized outer space and $\mathcal{MC}(\free)$, the space of measured currents. The precise statement is as follows: 
\begin{theorem}[{\cite[Theorem A]{KL:IntersectionForm}}]
There is a unique $\op{Out}(\free)$-invariant, continuous length pairing that is $\mathbb{R}_{\geq 0}$ homogeneous in the first coordinate and $\mathbb{R}_{\geq 0}$ linear in the second coordinate. 
$$ \la \cdot, \cdot \ra: \overline{cv}_n \times \mathcal{MC}(\free) \to \mathbb{R}_{\geq 0}$$
Further, $\la T, \eta_g\ra = l_T(g)$ for all $T \in \overline{cv_n}$ and all rational currents $\eta_g$ where $g \in \free \setminus \{1\}$. 
\end{theorem}

Kapovich and Lustig also give the following characterization of zero pairing:

\begin{proposition}[{\cite[Theorem 1.1]{KL:ZeroLength}}] Let $T \in \overline{cv}_n$, and let $\eta \in \mathcal{MC}(\free)$. Then $\la T, \eta \ra = 0$ if and only if $\op{supp}(\eta) \subseteq L(T)$, where $L(T)$ is the dual lamination of $T$ and $\op{supp}(\eta)$ is the support of $\eta$ in $\partial^2 \free$. \end{proposition}
  
In this section we want to define an intersection form for $\RelCVClo$, the closure of relative outer space and $\rc$, the space of relative currents. If $T \in \RelCVClo$ and  $\eta_{\alpha} \in \rc$ is a rational relative current then we can define $\la T, \eta_{\alpha} \ra:= l_T(\alpha)$ as in the absolute case. But unfortunately this length pairing is not continuous. The following example was shown to us by Camille Horbez. 

\begin{example}\label{E:CounterExample}Let $F_2 =\langle a,b\rangle$ with $\ffa = \{[\la a \ra]\}$. Let $T_k \in \RelCV$ be a simplicial tree such that $\Gamma_k = T_k/\free$ is a graph with two vertices joined by an edge and there is a loop at one of the vertices. Let $\langle a \rangle$ be the stabilizer of the vertex away from the loop. The graph $\Gamma_k$ is marked such that the loop is labeled by $a^kb$. Let the loop and the edge have length 1. The limit of the sequence of trees $T_k$ is the Bass-Serre tree of an HNN extension whose vertex stabilizer is given by $\langle a \rangle$ and it has a length 3 loop labeled $b$. Next consider a sequence of relative currents $\eta_k = \eta_{a^kb}$ converging to $\eta_{\infty}$, which is given by $\eta_{\infty}(a^nba^m) = 1$ for all $n,m \geq 0$ and $\eta_{\infty}(w) = 0$ for all other $w \in \free \setminus \ffa$. We have that $\la T_k, \eta_k\ra = 1$ and $\la T_k, \eta_{k+1}\ra = 3$ for all $k$. For continuity of the pairing, we would need $\la T_k, \eta_k \ra$ and $\la T_k, \eta_{k+1}\ra$ to converge to $\la T, \eta_{\infty}\ra$ but the limit doesn't exist in this example.
\end{example}

In fact, in \cite[Theorem 6]{GH:LaminationsFreeProducts}, Guirardel and Horbez show that there is no continuous pairing between the closure of relative outer space and the space of relative currents (their notion of relative currents is slightly different than ours). For the current purposes, in Section~\ref{subsec:IntersectionForm} we will define a pairing for $\relCVClo$ and $\prc$ along the lines of zero pairing criterion of Kapovich and Lustig and show that it has `enough continuity' (see Lemma~\ref{L:ContinuityAtZero}) for our application. Before we give the pairing map, we prove some lemmas about the dual lamination of trees and the support of relative currents. 
\vspace{.3cm}\\
Recall Notation~\ref{N:maps} and Definition~\ref{D: A TrainTrack} for the next two sections. 
\subsection{Dual lamination of a tree} The key results from this section which will be used later are Lemma~\ref{L:DualLamination}, Proposition~\ref{P:LaminationLimit} and Lemma~\ref{L:LaminationLimit}. Recall the definition of dual lamination of a tree from Section~\ref{subsec:Laminations}.

\begin{lemma}\label{L:DualLamination}
$\MPlaminations \subseteq L(\PMTrees)$ , $\PMlaminations \nsubseteq L(\PMTrees)$.    
\end{lemma}\begin{proof}
We have $\displaystyle{\StableTree = \lim_{n \to \infty} \frac{\RelTTTree \phi^n}{\PFevalue^n}}$. Let $w$ be a non-trivial conjugacy class in $[\free \setminus \ffa]$. Assume $l_{\StableTree}(w)=1$. Let $g_m = \oo^{-m}(w)$. Then $l_{\StableTree}(g_m) = 1 / \PFevalue^m$ which implies $(g_m^{-\infty}, g_m^{\infty})$ is contained in $L_{1/ \PFevalue^m}(\StableTree)$. Thus $l_- = \lim_{m \to \infty}g_m$ is contained in $\displaystyle{L(\StableTree) = \bigcap_{m \to \infty}L_{1/ \PFevalue^m}(\StableTree)}$. Since $l_-$ is a generic leaf of $\Rlamination$ and $L(\StableTree)$ is a closed subset of $\partial^2 \free$ we conclude that $\Rlamination \subseteq L(\StableTree)$.  

Let $g_m = \oo^m(w)$ such that $g_m$ converges to a generic leaf $l_+ \in \lamination$. We have $l_{\StableTree}(g_m) = \PFevalue^m l_{\StableTree}(w)$ which grows as $m$ goes to infinity. Thus $l_+ \notin L(\StableTree).$ \end{proof}

\begin{lemma}
The stable and unstable trees $\PMTrees$ have dense orbits. \end{lemma} \begin{proof}
By a result of \cite[Proposition 4.16]{H:HorbezBoundary} which is a relativization of Levitt's decomposition theorem for trees in $CV_n$ \cite{L:GraphOfActions} we have the following: if $\StableTree$ does not have dense orbits then $\StableTree$ splits uniquely as a graph of actions, all of whose vertex trees have dense orbits, such that the Bass–Serre tree $\mathcal{G}_{\StableTree}$ of the underlying graph of groups is very small (Section~\ref{subsec:RelativeOuterSpace}), and all its edges have positive length. Up to taking powers $\mathcal{G}_{\StableTree}$ is $\oo$-invariant. If $\mathcal{G}_{\StableTree}$ has an edge with trivial stabilizer then by collapsing all other edges we get a $\oo$-invariant free factor system, which is a contradiction. If the edge stabilizers are non-trivial, then they are non-peripheral. Then by theorems of Shenitzer \cite{Shenitzer} and Swarup \cite{S:DecompositionFreeGroups} there is a smallest non-trivial free factor system containing the edge stabilizer and $\ffa$, which will have to be $\oo$-invariant. This is a contradiction.      
\end{proof}

Consider a sequence of trees $T_k$ in $\overline{CV}_n$ converging to a tree $T$. Then we can look at the sequence of laminations $L(T_k)$ and ask if its limit is equal to $L(T)$ or not. An example in \cite[Section 9]{CHL:DualLaminationsII} shows that $L_{\infty} = \lim_{n \to \infty} L(T_n)$ may not be equal to $L(T)$. We record another example here. 


\begin{example}\label{E:Limit of laminations}
Let $\free = \la a, b \ra$ be the free group of rank two. Let $T_k$ be a simplicial $\free$-tree given as follows: it is the universal cover of the one edge free splitting with vertex stabilizers given by $\la a^k b\ra$ and $\la a \ra$. The sequence $T_k$ converges to a tree $T$ which is the Bass-Serre tree of the HNN extension with vertex group $\la a \ra$ and edge labeled $b$. The algebraic lamination $L(T_k)$ is the set of periodic lines determined by $a$ and $a^kb$ which converges to periodic lines determined by $a$, denoted $\ldots aaaa \ldots$, and lines of the form $\ldots aaaa \,b \, aaaa \ldots$. On the other hand, $L(T)$ is given by the periodic lines determined by $a$. We see that the birecurrent line in the limit of the laminations $L(T_k)$ is contained in $L(T)$. This is in fact always true by a result of \cite{CHL:LaminationLimit} (see Proposition~\ref{P:LaminationLimit}). 
\end{example}

We need the following lemma for the proof of Proposition~\ref{P:LaminationLimit}.   

\begin{lemma}\label{L:LaminationInclusion}
Let $T$ be a tree in $\overline{CV}_n$. Then the birecurrent leaves of $L_{\infty}(T)$, which is the algebraic lamination defined by the birecurrent laminary language associated to $L^1(T)$, are contained in $L(T)$. \end{lemma}
\begin{proof} We look at different cases according to whether $T$ is simplicial or has dense orbits. 
\begin{itemize}
\item \emph{$T$ has dense orbits:} by \cite[Proposition 5.8]{CHL:DualLaminationsII} a stronger statement is true, which says that $L_{\infty}(T) = L(T)$.
\item \emph{$T$ is simplicial with trivial edge stabilizers but is not free:} let $\hat{T}$ be a free simplicial tree with a collapse map $c: \hat{T} \to T$ and $BCC(c)$ equal to zero. The map $c$ extends to $\partial \hat{T}$ and we denote its restriction to $\partial \hat{T}$ by $Q: \partial \hat{T} \to T \sqcup \partial T$. If $X \in \partial \hat{T}$ is carried by a vertex stabilizer of $T$ then $Q(X)$ is precisely (since $c$ has no cancellation) the vertex in $T$ with that stabilizer, otherwise $Q(X)$ is a point in $\partial T$. Let $l=\{X, X^{\p}\}$ be a birecurrent leaf in $L_{\infty}(T)$. Since $X$ and $X^{\p}$ are $T$-bounded, $Q(X)$ and $Q(X^{\p})$ are vertices in $T$. If $Q(X) \neq Q(X^{\p})$ then $l$ crosses an edge $e$ in $\hat{T}$ that maps to a non-degenerate edge in $T$. Since $l$ is birecurrent, $l$ crosses translates of $e$ infinitely often which implies that $X$ or $X^{\p}$ is not $T$-bounded. Thus $Q(X)=Q(X^{\p})$. Thus $l$ is carried by a vertex stabilizer of $T$ and hence $l \in L(T)$. 

\item \emph{$T$ is simplicial with non-trivial edge stabilizers:} by results of \cite{S:DecompositionFreeGroups} and \cite{Shenitzer}, for $T$ there exists $\hat{T}$ a free simplicial tree with an $\free$-equivariant map $c: \hat{T} \to T$ which is a composition of a collapse map and a fold map. The edge paths in $\hat{T}$ that possibly backtrack under the map $c$ are the ones that cross a minimal subtree (of $\hat{T}$) of an edge stabilizer of $T$. By \cite[Lemma 3.1]{BFH:Laminations} $\op{BCC}(c) \leq \op{Lip}(c) \op{vol}(\hat{T})$. By scaling the metric on $\hat{T}$ we may assume that $\op{Lip}(c)$ is less than equal to 1. Since the volume of the free simplicial tree $\hat{T}$ is bounded, $\op{BCC}(c)$ is finite. 

As before consider the map $Q: \partial \hat{T} \to T \sqcup \partial T$. Let $X \in \partial \hat{T}$ be represented by a one-sided infinite word $x$ starting at the basepoint in $\hat{T}$. If the tail of $x$ is carried by a vertex stabilizer of $T$ then except an initial segment, $c(x)$ crosses the corresponding vertex in $T$ infinitely often with possibly some bounded backtracking. Thus we set $Q(X)$ to be that vertex. If the tail of $x$ is carried by an edge stabilizer $H$ then except an initial segment, $c(x)$ is a vertex of $T$ whose stabilizer contains $H$ and we set $Q(X)$ to be that vertex. Even though there are finitely many vertices in $T$ whose stabilizer contains $H$ there is only one minimal subtree for $H$ in $\hat{T}$, which maps to a unique vertex in $T$. Thus in this case $Q(X)$ only depends on the choice of $\hat{T}$. If the tail of $x$ is neither carried by a vertex stabilizer nor an edge stabilizer then $Q(X)$ is an element of $\partial T$. 

Now for a birecurrent leaf $l=\{X, X^{\p}\}$ such that $X$ and $X^{\p}$ are $T$-bounded, we get that $Q(X)=Q(X^{\p})$. Thus the leaf $l$ maps to a vertex of $T$ under the map $c$ with possiblly bounded backtracking from edges in $\hat{T}$ that fold under the map $c$. Hence $l$ is in $L(T)$.        


\item \emph{When $T$ is neither simplicial nor does it have dense orbits:} let $T^{\p}$ be the simplicial tree which is the graph of actions (see \cite{G:LimitGroups} for definition) of $T$ corresponding to the Levitt decomposition \cite{L:GraphOfActions} of $T$. Let $l=\{X, X^{\p}\}$ be a birecurrent leaf in $L_{\infty}(T)$. Since $X, X^{\p} \in L^1(T)$, we get that $X, X^{\p}$ are also $T^{\p}$-bounded. Since $l$ is birecurrent, by the previous two cases we get that $l$ is carried by a vertex stabilizer $H$ of $T^{\p}$. We are interested in the vertices of $T^{\p}$ that correspond to subtrees with dense orbits in $T$. Thus we assume that $l$ is contained in some subtree $T_d$ of $T$ with dense orbits and stabilizer $H$. Since $T_d$ is a subtree of $T$, $X$ and $X^{\p}$ are also $T_d$-bounded. 

The subgroup $H$ is finitely generated because point stabilizers in the very small tree $T^{\p}$ have bounded rank \cite{GL:RankOfActions}. Therefore, there exists a finite graph $\Gamma_H$ and an immersion $i: \Gamma_H \to R_{\mathfrak{B}}$, where $R_{\mathfrak{B}}$ is a rose with petals labeled by elements of a basis $\mathfrak{B}$ of $\free$, such that $\pi_1(i(\Gamma_H)) = H$. Since $H$ carries $l$, which can be viewed as a map $l: \mathbb{Z} \to R_{\mathfrak{B}}$, there exists a map $l_H: \mathbb{Z} \to \Gamma_H$ such that $i \circ l_H = l$. Since $l$ is birecurrent we claim that $l_H$ is also birecurrent. Consider a word $w$ in $l_H$ such that $i(w)$ is a subword of $l$. Since $l$ is birecurrent $i(w)$ appears infinitely often in both ends of $l$. Let $w_1, w_2, \ldots, w_n$ be the pre-images of all occurrences of $i(w)$ in $l_H$. There are only finitely many such $w_i$ because $\Gamma_H$ is a finite graph. Thus at least one of the $w_i$ appears infinitely often in both ends of $l_H$. But we need to show that every such $w_i$ appears infinitely often in $l_H$. So consider a finite subword $u$ of $l_H$ that contains at least one appearance of each $w_i$. Such a word exists because there are only finitely many $w_i$. Now $i(u)$ appears infinitely often in both ends of $l$. Therefore, some pre-image $u_1$ of $i(u)$ in $l_H$ appears infinitely often. Since every pre-image of $i(u)$ contains all the $w_i$s we get that each $w_i$ appears infinitely often in both ends of $l_H$. Thus $l_H$ is birecurrent.          

Let $l_H = \{X_H, X_H^{\p}\}$. Since $i$ is an immersion and $X, X^{\p}$ are $T_d$-bounded we have $X_H, X_H^{\p}$ are also $T_d$-bounded. Thus $l_H$ is in $L_{\infty}(T_d)$, which is equal to $L(T_d)$ by the first case. Since $T_d$ is a subtree of $T$ and $l$ is contained in $T_d$, we get that $l=i\circ l_H$ is in $L(T)$. 
\qedhere
\end{itemize}
\end{proof}

\begin{definition}
A lamination $L$ is called birecurrent if every leaf of $L$ is birecurrent. 
\end{definition}

\begin{proposition}[{\cite{CHL:LaminationLimit}}]\label{P:LaminationLimit}
Let $\{T_k\}_{k \in \mathbb{N}}$ be a sequence of trees in $\overline{CV}_n$ converging to a tree $T$. Also suppose that the sequence of laminations $L(T_k)$ converges to $L_{\infty}$ in $\Lambda^2(\free)$. Let $L_r$ be a birecurrent sublamination of $L_{\infty}$. Then $L_r \subseteq L(T)$.  \end{proposition}
For completeness we give a proof of the above proposition. 
\begin{proof}
We will use notation from \cite{CHL:DualLaminationsII}. If the trees $T_k$ are free simplicial then their dual lamination is empty and the lemma is true vacuously. So let's assume that $L(T_k)$ is non-empty. Let $l=\{X, X^{\p}\}$ be a leaf of $L_{\infty}$. Fix a basis $\mathfrak{B}$ of $\free$ and realize $X$ in this basis as a one-sided infinite word. For $l \geq 1$, let $X_l \in \free$ be the prefix of length $l$ of $X$. We first show that $X \in L^1(T)$, that is, for a point $p \in T$ the sequence $X_l p$ is bounded in $T$. Suppose not. Then for any $C>0$, $p \in T$, $K_0>0$, there exists $q>r>K_0$ such that $d_T(X_q p, X_r p)>C$. Let $u = X_r^{-1}X_q$. Then $d_T(up, p)>C$. By Gromov-Hausdorff topology on $\overline{CV}_n$, given $p , up \in T$, let $p_k, s_k\in T_k$ be approximations of $p$ and $up$ relative to some exhaustions (see \cite[Lemma 4.1]{H:SpectralRigidity} for details). Then $d_{T_k}(up_k, s_k)$ goes to zero and $d_{T_k}(s_k, p_k)$ goes to $d_{T}(up, p)$ as $k \to \infty$. Thus given $\delta>0$ there exists a $K_1>0$ such that for all $k > K_1$, $d_T(up, p) - \delta \leq d_{T_k}(up_k, p_k)$, or in other words, $d_{T_k}(up_k, p_k) \geq C-\delta$. 

Now by the convergence criterion (Definition~\ref{D:ConvergenceLaminations}), for any $m\geq 1$ there exists a $K_2(m)>0$ such that for all $k \geq K_2$, $\mathcal{L}_m(L(T_k)) = \mathcal{L}_m(L_{\infty})$. Let $m$ be the word length of $u$ with respect to the fixed basis. Since $u \in \mathcal{L}_m(L_{\infty})$ we get that $u \in \mathcal{L}_m(L(T_k))$ for all $k > \op{max}(K_0, K_1, K_2)$. By \cite[Remark 4.2]{CHL:DualLaminationsII}, this means that for every $\epsilon >0$ there exists a cyclically reduced $w$ in $\free$ such that $||w||_{T_k} < \epsilon$ and $u$ is a subword of $w$. Also by \cite[Lemma 3.1(c)]{CHL:DualLaminationsII} $$ d_{T_k}(up_k, p_k) \leq 2\op{BCC}(\mathfrak{B}, p_k)+||w||_{T_k},$$ where $\op{BCC}(\mathfrak{B}, p_k)$ is the bounded cancellation constant of the $\free$-equivariant map from $\op{Cay}(\mathfrak{B})$ to $T_k$ such that the base point of $\op{Cay}(\mathfrak{B})$ is mapped to $p_k$. We claim that $\op{BCC}_k:=\op{BCC}(\mathfrak{B}, p_k)$ is bounded above by a constant. Let $\op{BCC}_T:=\op{BCC}(\mathfrak{B}, p)$. Since $up$ is in the $\op{BCC}_T$ neighborhood of an axis of $w$ in $T$ then by \cite[Lemma 4.1 (c)]{H:SpectralRigidity} for sufficiently large $k$, $s_k$ is in the $\op{BCC}_T+1$ neighborhood of axis of $w$ in $T_k$. Given $\delta^{\p}>0$, for sufficiently large $k$, $d_{T_k}(up_k, s_k) \leq \delta^{\p}$. Therefore, $up_k$ is in a $\op{BCC}_T+1+\delta^{\p}$ neighborhood of axis of $w$ in $T_k$. Since this is true for any cyclically reduced word $w$ and a subword $u$ we get that $BCC_k \leq \op{BCC}_T+1+\delta^{\p}$.

By choosing $C$ large enough we get a contradiction since 
$$C-\delta \leq d_{T_k}(up_k, p_k) \leq 2BCC_k+||w||_{T_k} \leq 2(BCC_T+1+\delta^{\p}) + \epsilon$$ for all $k$ sufficiently large. Thus we have that $X$ and similarly $X^{\p}$ are both in $L^1(T)$. Since $l$ is birecurrent, we get $l \in L_{\infty}(T)$. 

Now if $l = \{X, X^{\p}\}$ is birecurrent and $l \in L_{\infty}(T)$ then by Lemma~\ref{L:LaminationInclusion}, $l \in L(T)$. Thus $L_r \subseteq L(T)$. 
\end{proof}

\begin{lemma}\label{L:LaminationLimit}
Let $\{T_k\}_{k \in \mathbb{N}}$ be a sequence of trees in $\overline{CV}_n$ converging to a tree $T$ such that $T$ has dense orbits. Also suppose that the sequence of laminations $L(T_k)$ converges to $L_{\infty}$ in $\Lambda^2(\free)$. Then $L_{\infty} \subseteq L(T)$.  \end{lemma}
\begin{proof} If the trees $T_k$ are free simplicial then $L_{\infty} = \emptyset$. Thus after passing to a subsequence we assume that $L(T_k) \neq \emptyset$. Since $T$ has dense orbits, by \cite[Proposition 2.2]{LL:NSDynamics} (see Proposition~\ref{P:SmallBBT}), given $\epsilon >0$ there exists a free simplicial $\free$-tree $S$ and an $\free$-equivariant map $h: S \to T$ which is isometric on edges ($\op{Lip}(h) = 1$) and $\op{BCC}(h) < \op{vol}(S) <\epsilon$. We will now construct $\free$-equivariant maps $h_k: S \to T_k$ for $k$ sufficiently large such that $\op{BCC}(h_k)$ is bounded above by a linear function of $\epsilon$. 

For trees $S \in CV_n$ and $T$ in $\overline{CV}_n$, let $\op{Lip}(S, T)$ be the infimum of the Lipschitz constant of all $\free$-equivariant maps $f:S \to T$. By \cite[Proposition 4.5]{AK:MetricCompletion}, \cite[Theorem 0.2]{H:SpectralRigidity}, $\op{Lip}(S, T)$ is equal to $$\Lambda(S, T) := \sup_{g \in \free \setminus \{1\}}\frac{||g||_{T}}{||g||_S}. $$ By \cite[Proposition 4.5]{AK:MetricCompletion}, \cite[Proposition 6.15, 6.16]{H:SpectralRigidity} the supremum above can be taken over a set of candidates $\mathcal{C}(S) \subset \free$. Since $S$ is free simplicial, the set $\mathcal{C}(S)$ is finite. 

For every $\delta >0$ and the finite set $\mathcal{C}(S)$ of elements of $\free$, there exists a $K>0$ such that for all $k\geq K$ and for all $g \in \mathcal{C}(S)$, $$||g||_{T_k} \leq ||g||_T+\delta.$$

Thus we have that $\Lambda(S, T_k) \leq \Lambda(S, T) + \delta^{\p}$ where $\delta^{\p}$ is the maximum of $\delta / ||g||_S$ over all $g \in \mathcal{C}(S)$. This implies that $\op{Lip}(S, T_k) \leq \op{Lip}(S, T) + \delta^{\p} \leq \op{Lip}(h) + \delta^{\p} \leq 1+\delta^{\p}$. By \cite[Theorem 0.4]{H:SpectralRigidity} $\op{Lip}(S,T_k)$ is realized, that is, there exists an $\free$-equivariant map $h_k: S \to \overline{T_k}$, where $\overline{T_k}$ is the closure of $T_k$, such that $\op{Lip}(h_k) = \op{Lip}(S, T_k)\leq 1+\delta^{\p}$ for all $k \geq K$. Also $$\op{BCC}(h_k) \leq \op{Lip}(h_k) \op{vol}(S) \leq (1+\delta^{\p})\epsilon.$$

Now consider a sequence of leaves $l_k \in L(T_k)$ converging to a leaf $l \in L_{\infty}$. Then by Proposition~\ref{P:Q map} ($\mathcal{Q}$ map), the diameter of $h_k(l_k)$ in $\overline{T_k}$ is bounded by $ 2 \op{BCC}(h_k)$ which is less than $2(1+\delta^{\p}) \epsilon$. Thus, in the limit, the diameter of $h(l)$ in $\overline{T}$ is bounded above by $2(1+\delta^{\p})\epsilon$. Since $\epsilon$ and $\delta$ were arbitrary, we get that $l \in L(T)$.
\end{proof}
\subsection{Support of a relative current}
Support of a relative current $\eta$ is defined as the closure in $\mathbf{Y}$ (see Section~\ref{subsec:RelativeCurrents} for definition) of the intersection of the complement of all open sets $U \subset \mathbf{Y}$ such that $\eta(U)=0$. For $\eta \in \prc$, $\op{supp}(\eta)$ is a closed, non-empty and $\free$-invariant subset of $\mathbf{Y}$. Since $\mathbf{Y}$ is not a closed subset of $\partial^2 \free$, $\op{supp}(\eta) \subset \mathbf{Y}$ may not be a closed subset of $\partial^2 \free$. Let $\overline{\op{supp}(\eta)}$ denote its closure in $\partial^2\free$. Then $\overline{\op{supp}(\eta)} \setminus \op{supp}(\eta)$ is contained in $\partial^2 \ffa$ which is non-empty when lines in $\op{supp}(\eta)$ accumulate on lines in $\partial^2 \ffa$.    

\begin{example}\label{E:CurrentSequence} Let $F_2 = \la a, b \ra$, $\ffa = \{[\la a \ra]\}$ and consider the sequence of relative currents $\eta_{a^nb}$ converging to $\eta_{\infty}$ in $\prc$ as in Example~\ref{E:CounterExample}. Then $\op{supp}(\eta_{\infty})$ is given by bi-infinite geodesics determined by $\ldots aaa \, b \, aaa \ldots$. Thus the set $\overline{\op{supp}(\eta_{\infty})}$ also contains the bi-infinite lines given by $\ldots aaaa \ldots$. 
Geometrically, consider a lamination $L$ on a torus with one puncture (with fundamental group identified with $F_2 = \la a, b \ra$) as follows: the lamination $L$ contains the simple closed curve $a$ and another leaf $l$ which goes around $b$ and spirals towards $a$ from both sides. In the absolute case, the support of the current $\mu_{a^kb}$ are the curves $a$, $b$ and the curve $c_k$ obtained by Dehn twisting $b$ around $a$, $k$ times. The absolute currents $\mu_{a^kb}$ projectively converge to the absolute current $\mu_{a}$ whose support is just the curve $a$. But in the relative case, the support of the relative current $\eta_{a^kb}$ is the curve $c_k$ and the relative currents $\eta_{a^kb}$ converge to $\eta_{\infty}$ whose support is the leaf $l$. Thus when we take the closure of $l$ we also get the curve $a$.  \end{example}

We have that $\overline{\op{supp}(\eta)}$ is a closed, non-empty, $\free$-invariant subset of $\partial^2\free$. Recall Notation~\ref{N:maps} for a relative train track representative of $\oo$. 

\begin{lemma}\label{L:MinimalLamination}
$\lamination \cap \mathbf{Y}$ is minimal in $\mathbf{Y}$, that is, $\lamination \cap \mathbf{Y}$ contains no proper closed (in $\mathbf{Y}$), non-empty $\free$-invariant subset.\end{lemma}
\begin{proof} By \cite[Lemma 3.1.15]{BFH:Tits} we have the following: suppose $\delta$ is a generic leaf in $\lamination$ that is not entirely contained in $G_{r-1}$. Then the closure of $\delta$ in $\partial^2 \free$ is all of $\lamination$. Suppose $\lamination \cap \mathbf{Y}$ contains a proper, closed (in $\mathbf{Y}$), non-empty, $\free$-invariant subset $S$. A generic leaf $\delta$ in $S$ is not entirely contained in $G_{r-1}$ where $\mathcal{F}(G_{r-1}) = \ffa$. Since $\mathbf{Y}$ gets subspace topology from $\partial^2\free$, the closure of $\delta$ in $\mathbf{Y}$ is all of $\lamination \cap \mathbf{Y}$ which is a contradiction. 
\end{proof}

\begin{lemma}\label{L:Support} We have 
\begin{enumerate}[(a)]
 \item $\op{supp}(\PMCurrents)$ as a subset of $\mathbf{Y}$ is equal to $\PMlaminations \cap \mathbf{Y}$. 
 \item $\overline{\op{supp}(\PMCurrents)} \subseteq \PMlaminations \cup \partial^2 \ffa$.
\end{enumerate} \end{lemma}

A proof of a similar fact in the case of a fully irreducible automorphism can be found in \cite[Proposition 6.1]{CP:Twisting}. 

\begin{proof}
Let $a$ be a primitive conjugacy class in $[\free \setminus \ffa]$ realized as $\alpha$ in $G^{\p} = T_G^{\p} / \free$. Then $\alpha$ is a union of $N$ $r$-legal paths for some $N >0$. 
For every $m\geq 0$, $\alpha_m:=(\phi^{\p})^m(\alpha)$ contains at most $N$ segments of leaves of $\lamination \cap \mathbf{Y}$. Let the complement of $\lamination \cap \mathbf{Y}$ in $\mathbf{Y}$ be covered by cylinder sets $C(\gamma)$ where $\gamma$ is a subpath of $G^{\prime}$ that crosses $H_r$ and is not crossed by any leaf of $\lamination$. For every $m \geq 0$, $\alpha_m$ contains at most $N$ occurrences of $\gamma$ (at concatenation points of the $r$-legal segments). Thus $\eta_{\alpha_m}(C(\gamma)) \leq N$. Since $\eta_{\alpha_m}/\PFevalue^m \to \StableCurrent$ as $m \to \infty$, we have that $\StableCurrent(C(\gamma)) = 0$. Thus $\op{supp}(\StableCurrent) \subseteq \lamination \cap \mathbf{Y}$. By Lemma~\ref{L:MinimalLamination}, $\lamination \cap \mathbf{Y}$ is minimal in $\mathbf{Y}$ therefore we have $\op{supp}(\StableCurrent) = \lamination \cap \mathbf{Y}$. Since $\lamination$ is a closed subset of $\partial^2\free$ we get that $\overline{\op{supp}(\StableCurrent)} \subseteq \lamination \cup \partial^2 \ffa$. 
\end{proof}

\begin{lemma}\label{L:CurrentLimit}
Let $\{\eta_n\}_{n \in \mathbb{N}}$ be a sequence of relative currents converging to a relative current $\eta$. Suppose the sequence $\op{supp}(\eta_n)$ converges to $S \subset \mathbf{Y}$. Then $\op{supp}(\eta) \subseteq S$. \end{lemma} 
\begin{proof}
Consider a word $w \in \free \setminus \ffa$ such that $\eta(w)>0$. Then given $\epsilon>0$ there exists an $N_0>0$ such that for every $n>N_0$, $\eta_n(w) > \epsilon$. Thus $C(w) \cap \op{supp}(\eta_n)$ is non-empty for every $n \geq N_0$ which implies that $C(w) \cap S$ is non-empty. Since this is true for any word $w \in \free \setminus \ffa$ with $\eta(w)>0$, we get that $\op{supp}(\eta) \subset S$. 
\end{proof} 
\subsection{Intersection form}\label{subsec:IntersectionForm}
\begin{definition}
Define a function $I: \relCVClo \times \prc \to \{0,1\}$ as follows: 
$$I(T,\eta) = 0 \text{ if } \overline{\op{supp}(\eta)} \subseteq L(T),$$
$$I(T,\eta) = 1 \text{ if } \overline{\op{supp}(\eta)} \nsubseteq L(T).$$ 
\end{definition}

\begin{lemma}\label{L:ContinuityAtZero}
The function $I$ satisfies the following properties:
\begin{enumerate}[(a)]
\item $I(T \Psi, \eta) = I(T, \Psi\eta)$ for $\Psi \in \op{Out}(\free, \ffa)$. 
\item Let $T_k \to T$ in $\relCVClo$ and $\eta_k \to \eta$ in $\prc$ such that $I(T_k, \eta_k) = 0$ for all $k$. If either $T$ has dense orbits or $\overline{\op{supp}(\eta)}$ is a birecurrent lamination then $I(T, \eta) = 0$. 
\end{enumerate}\end{lemma}

It is not true in general that if $I(T_k, \eta_k) = 0$ for all $k$ then $I(T, \eta) = 0$. Consider the sequence of trees $T_k$ as in Example~\ref{E:Limit of laminations} and the sequence of currents $\eta_k$ as in Example~\ref{E:CurrentSequence}. Then $I(T_k, \eta_k)=0$ but $I(T, \eta) \neq 0$.  
\begin{proof}
\begin{enumerate}[(a)]
\item We have $\op{supp}(\Psi\eta) = \Psi \op{supp}(\eta)$ and $L(T\Psi) = \Psi^{-1}L(T)$ which gives the desired equality.
\item Let $\mathcal{S}$ be the closure of $\lim_{n \to \infty}\op{supp}(\eta_n)$ and let $L(T_n)$ converge to $L_{\infty}$. Then $\mathcal{S} \subseteq L_{\infty}$ and by Lemma~\ref{L:CurrentLimit}, $\overline{\op{supp}(\eta)} \subseteq \mathcal{S}$. If $T$ has dense orbits then by Lemma~\ref{L:LaminationLimit}, $L_{\infty} \subseteq L(T)$. Thus $\overline{\op{supp}(\eta)} \subseteq L(T)$. If $\overline{\op{supp}(\eta)}$ is a birecurrent lamination then by Proposition~\ref{P:LaminationLimit}, it is contained in $L(T)$.
\qedhere \end{enumerate}\end{proof}

\begin{lemma}[Uniqueness of dual]\label{L:Uniqueness of dual}
Let $\oo$ be a fully irreducible outer automorphism relative to $\ffa$. Let $T\in \relCVClo$ and $\eta \in \prc$. Then 
\begin{enumerate}[(a)]
\item $I(\PMTrees, \MPCurrents) = 0$
\item If $I(\PMTrees, \eta) = 0$ then $\eta = \MPCurrents$. 
\item If $I(T, \MPCurrents) = 0$ then $T= \PMTrees$.  
\end{enumerate} \end{lemma}
\begin{proof}\begin{enumerate}[(a)]
\item By Lemma~\ref{L:DualLamination}, $\MPlaminations \subset L(\PMTrees)$. Also $\partial^2 \ffa \subset L(\PMTrees)$ because $\ffa$ is elliptic in $\PMTrees$. Thus by Lemma~\ref{L:Support}, $\overline{\op{supp}(\MPCurrents)} \subseteq L(\PMTrees)$. 
\item By Lemma~\ref{L:DualLamination} and Lemma~\ref{L:Support}, $\op{supp}(\StableCurrent) \nsubseteq L(\StableTree)$, therefore $I(\StableTree, \StableCurrent) \neq 0$. Now suppose $I(\StableTree, \eta) = 0$ for some $\eta \neq \UnstableCurrent$. Then by definition $\overline{\op{supp}(\eta)} \subseteq L(\StableTree)$. By the $\op{Out}(\free, \ffa)$ action we also get that $\overline{\op{supp}(\oo^n(\eta))} \subseteq L(\StableTree)$. By Theorem B, $\oo^n(\eta)$ converges to $\StableCurrent$ therefore in the limit $\overline{\op{supp}(\StableCurrent)} \subseteq L(\StableTree)$ which is a contradiction.  
\item Similar argument as above using Theorem C. 
\qedhere \end{enumerate}\end{proof}
\section{Proof of main theorem}\label{sec:MainProof}
We will now present a proof of Theorem A. 

\begin{lemma}[UV-pair]\label{L:UV pair}
Let $\oo$ be fully irreducible relative to  $\ffa$. For every neighborhood $U$ of $\StableTree$ in $\relCVClo$, $\exists$ a neighborhood $V$ of $\UnstableCurrent$ in $\prc$ such that for every $T\in U^C$ and $\eta \in V$ we have $I(T, \eta) \neq 0$. \end{lemma} 
\begin{proof}
Assume by contradiction that there exists a $U$ such that for every neighborhood $V$ of $\UnstableCurrent$ there exist $T \in U^C$ and $\eta \in V$ such that $I(T,\eta) =0$.

Let $V_i$ be an infinite sequence of nested neighborhoods of $\UnstableCurrent$ such that $V_i \supset V_{i+1}$ and $\cap V_i = \UnstableCurrent$. Then by assumption there exist $T_i \in U^C$ and $\eta_i \in V_i$ such that $I(T_i, \eta_i) = 0$. Since $\relCVClo$ is compact, after passing to a subsequence we have $T_i \to T$, for $T\neq \StableTree$. Also $\eta_i \to \UnstableCurrent$. Since the support of $\UnstableCurrent$ gives a birecurrent lamination, by Lemma~\ref{L:ContinuityAtZero} we get $I(T, \UnstableCurrent)= 0$, which contradicts Lemma~\ref{L:Uniqueness of dual}. \end{proof} 

\begin{lemma}[VU-pair]\label{L:VU pair}
For every neighborhood $V$ of $\UnstableCurrent$ in $\prc$, $\exists$ a neighborhood $U$ of $\StableTree$ in $\relCVClo$ such that for every $\eta \in V^C$ and $T\in U$ we have $I(T, \eta) \neq 0$.  \end{lemma} 
\begin{proof} Same as for Lemma~\ref{L:UV pair}. \end{proof}

\begin{lemma}\label{L:UVSequence}
There exist nested sequences 
$U_0 \supset U_1 \supset U_2 \supset U_3 \ldots \supset U_{2N}$ and $V_1 \supset V_2 \supset V_3 \ldots \supset V_{2N}$ of neighborhoods of $\StableTree$ and $\UnstableCurrent$ respectively such that the following are true: 
\begin{itemize}
\item $\exists \,k>0$ such that for every $i$, $\oo^k(U_{i}) \subset U_{i+1}$ and $\oo^{-k}(V_{i})\subset V_{i+1}$. 
\item $(U_i, V_{i+1})$ form a UV-pair for all $i \geq 0$. 
\item $(V_i, U_{i})$ form a VU-pair for all $i \geq 1$. 
\end{itemize} \end{lemma}
\begin{proof}
To build such sequences start with $U_0$. Then there exists $V_1$ such that $(U_0, V_1)$ form a UV-pair. Next there exists a $U_1$ such that $(V_1, U_1)$ form a VU-pair. If $U_1 \nsubseteq U_0$ then we can replace $U_1$ by a smaller open set in $U_1 \cap U_0$.

Let $r_i = \text{min}\{\text{p} \,|\, \oo^p(U_i) \subset U_{i+1}\}$ for $0 \leq i \leq 2N$ and let $s_i = \text{min}\{\text{p}\, | \, \oo^{-p}(V_i) \subset V_{i+1}\}$ for $0<i < 2N$. The numbers $r_i$ and $s_i$ exist because we have uniform north-south dynamics. Now define $k$ to be the maximum of the numbers $\{r_i\}_{i=0}^{2N}, \{s_i\}_{i=1}^{2N}$. 
\end{proof}

\begin{thmA}
Let $\ffa$ be a non-exceptional free factor system in a finite rank free group $\free$ of rank at least 3 and let $\oo \in \op{Out}(\free, \ffa)$. Then $\oo$ acts loxodromically on $\rfa$ if and only if $\oo$ is fully irreducible relative to $\ffa$. \end{thmA}

\begin{proof}
Let $\dd \in \rfa$ be a free factor system. Let $T_{\dd} \in \relCV$ be a simplicial tree such that its set of vertex stabilizers is equal to $\dd$. Let $\eta_{\dd}$ be a relative current with support contained in $\partial^2\dd$. Consider nested neighborhoods $U_0 \supset U_1 \supset \ldots \supset U_{2N}$ of $\StableTree$ and $V_1 \supset V_2 \supset \ldots \supset V_{2N}$ of $\UnstableCurrent$ and constant $k$ as in Lemma~\ref{L:UVSequence} such that $T_{\dd} \in U_0 \cap U_1^C$ and $\eta_{\dd} \in V_1^C$. See Figure~\ref{F:MainProof}. By Lemma~\ref{L:UV pair} and \ref{L:VU pair} we have the following: 
\begin{itemize}
\item If $T \in U_{i}^C$ and $I(T, \eta) = 0$ then $\eta \in V_{i+1}^C$. 
\item If $\eta \in V_{i}^C$ and $I(T, \eta) = 0$ then $T \in U_{i}^C$.
\end{itemize}
We have $T_{\dd}\oo^{ik} \in U_{i}$ and $\oo^{-ik}\eta_{\dd} \in V_i$. If $\dd$ is the set of vertex stabilizers of $T_{\dd}$ then $\oo^{-2ik}(\dd)$ is the set of vertex stabilizers of $T_{\dd}\oo^{2ik}$. 

\begin{figure}[h]
\centering
\includegraphics[scale=.5]{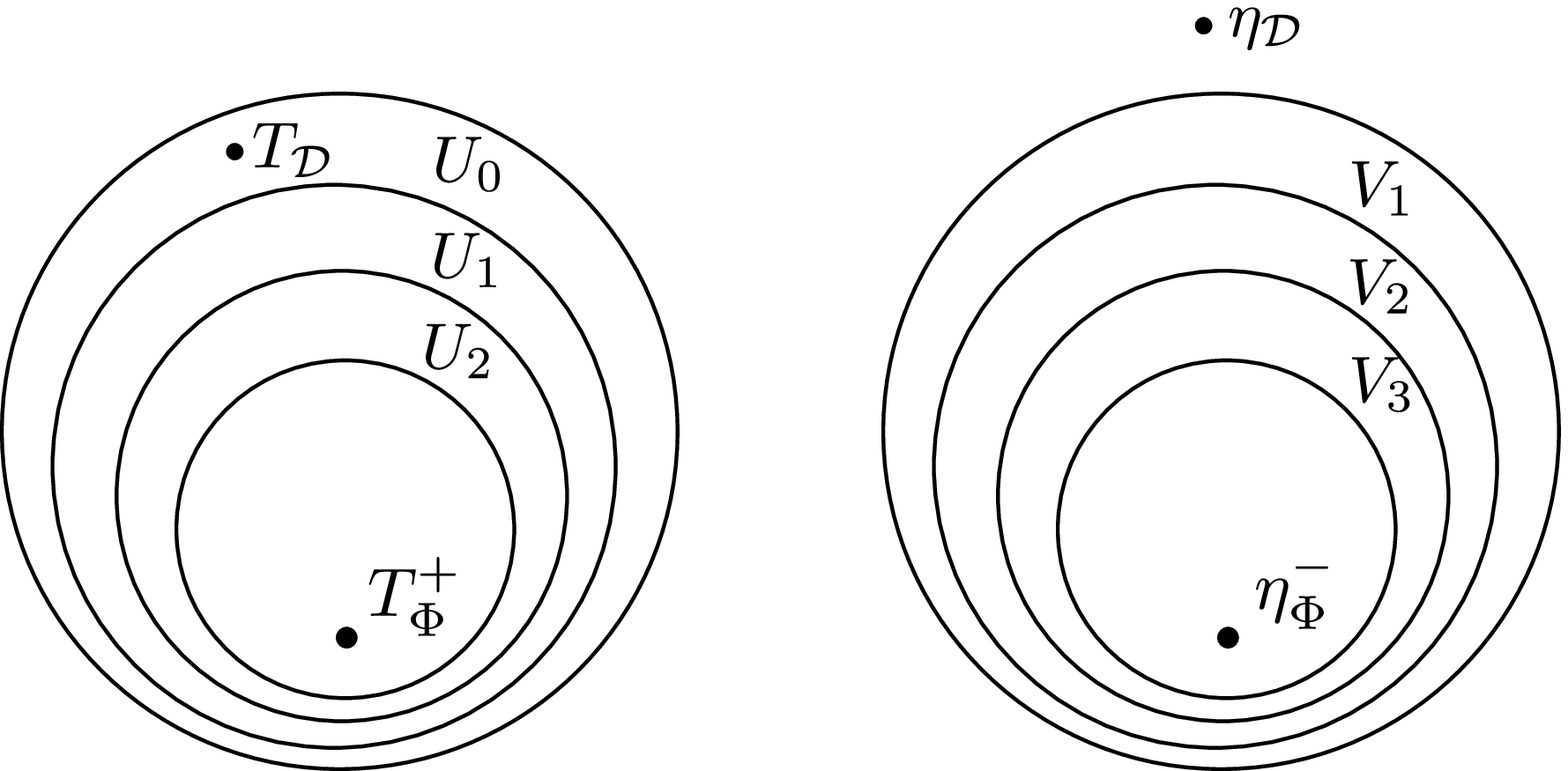}
\caption{}
\label{F:MainProof}
\end{figure}

We claim that $d_{\rfa}(\dd, \oo^{-2Nk}\dd) > 2N$ and $d_{\rfa}(\dd, \oo^{2Nk}\dd) > 2N$. For simplicity let's first consider the case when $N=1$ and for contradiction assume that $d_{\rfa}(\dd, \oo^{-2k}\dd) = 2$. 
Let $\ee$ be a free factor system distance one from both $\dd$ and $\oo^{-2k}\dd$. There are two cases to consider: 
\begin{enumerate}[(a)]
\item $\ee \sqsubset \dd$ and $\ee \sqsubset \oo^{-2k}\dd$: let $T_{\ee}$ be a simplicial tree whose set of vertex stabilizers is given by $\ee$. Choose $\eta$ such that $I(T_{\ee}, \eta) = 0$. Then $I(T_{\dd}, \eta) = 0$. Since $T_{\dd} \in U_1^C$ we get that $\eta \in V_2^C$. Also $I(T_{\dd}\oo^{2k}, \eta) = 0$ and since $\eta \in V_2^C$ we get $T_{\dd}\oo^{2k} \in U_2^{C}$. But that is a contradiction since $T_{\dd}\oo^{2k} \in U_2$. 
\item $\ee \sqsupset \dd$ and $\ee \sqsupset \oo^{-2k}\dd$:  we have that $I(T_{\ee}, \eta_{\dd}) = 0$. Since $\eta_{\dd} \in V_1^C$ we get $T_{\ee} \in U_1^{C}$. We also have that $I(T_{\ee}, \oo^{2k}\eta_{\dd}) = 0$. Since $T_{\ee} \in U_1^C$ we get $\oo^{-2k}\eta_{\dd} \in V_2^C$, which is a contradiction.  
\end{enumerate}

The above proof in particular also shows that $d_{\rfa}(\dd, \oo^{-2Nk}(\dd))>2$. Now for  any $N>0$, assume by contradiction that $d_{\rfa}(\dd, \oo^{-2Nk}\dd) \leq 2N$. Consider a geodesic $\dd = \ee_0, \ee_1, \ee_2 \ldots, \ee_{l}, \ee_{l+1} = \oo^{-2Nk}\dd$, $l<2N$, in $\rfa$. Without loss of generality assume that $\ee_1 \sqsubset \dd$. Then starting with applying the same argument as in $(a)$ for the triple $\dd, \ee_1, \ee_2$ we alternatively apply $(a)$ and $(b)$ to reach a contradiction. \end{proof}

As an example to exhibit the above proof for $N>1$, consider a geodesic $\dd = \ee_0, \ee_1, \ee_2 \ldots, \ee_{5}, \ee_{6} = \oo^{-6k}\dd$ in $\rfa$ connecting $\dd$ and $\oo^{-6k}\dd$. Without loss of generality assume that $\ee_1 \sqsubset \dd$. Let $T_i$ be a tree in $\relCVClo$ whose set of vertex stabilizers is given by $\ee_i$. We have $T_0 \in U_0 \cap U_1^C$ and thus $T_6$ is contained in $U_6$.  

\begin{itemize}
\item Given $T_1$, choose $\eta_1$ such that $I(T_1, \eta_1) = 0$ which implies that $I(T_{\dd}, \eta_1) = 0$ ($\op{supp}(\eta_1) \subset \partial^2\ee_1 \subset \partial^2\dd$). Also $I(T_2, \eta_1) =0$ because $\op{supp}(\eta_1) \subset \partial^2\ee_1 \subset \partial^2\ee_2$. 
\item Given $T_3$, choose $\eta_2$ such that $I(T_3, \eta_2) = 0$ which implies that $I(T_2, \eta_2) = 0$ ($\op{supp}(\eta_2) \subset \partial^2\ee_3 \subset \partial^2\ee_2$). Also $I(T_4, \eta_2) = 0$ because $\op{supp}(\eta_2) \subset \partial^2\ee_3 \subset \partial^2\ee_4$. 
\item Given $T_5$, choose $\eta_3$ such that $I(T_5, \eta_3) = 0$ which implies that $I(T_4, \eta_3) = 0$ ($\op{supp}(\eta_3) \subset \partial^2\ee_5 \subset \partial^2\ee_4$). Also $I(T_6, \eta_3) = 0$ because $\op{supp}(\eta_3) \subset \partial^2\ee_5 \subset \partial^2\ee_6$. 
\end{itemize}
Now using all of the above information we get $$T_{\dd} \in U_1^C \implies \eta_1 \in V_2^C \implies T_2 \in U_2^C \implies \eta_2 \in V_3^C \implies T_4 \in U_3^C \implies \eta_3 \in V_4^C \implies T_6 \in U_4^C,$$ which is a contradiction. See Figure~\ref{F:MainProofEx}.

\begin{figure}[h]
\centering
\includegraphics[scale=.5]{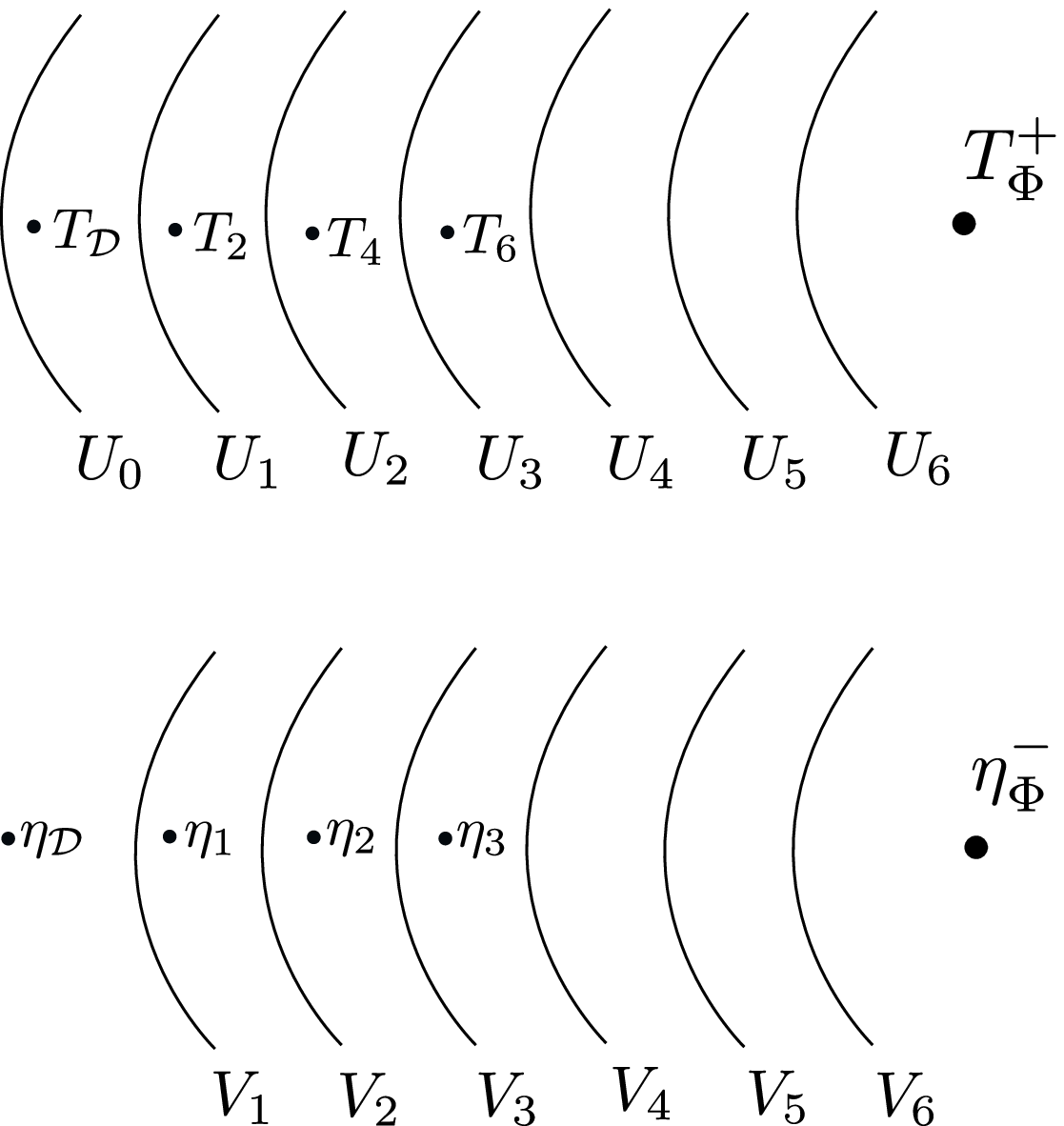}
\caption{}
\label{F:MainProofEx}
\end{figure}

\newpage
\bibliographystyle{alpha}
\bibliography{bib2}
\end{document}